\documentclass[11pt]{amsart}
\usepackage[all]{xy}
\usepackage[all]{xy} 
\usepackage{amssymb} 
\usepackage{amsmath} 
\usepackage{graphicx} 
\usepackage[english]{babel} 
\usepackage{epsfig}
\usepackage{color}
\usepackage{enumerate}
\swapnumbers

\theoremstyle{plain}
\newtheorem{theorem}[subsection]{Theorem}
\newtheorem{proposition}[subsection]{Proposition}
\newtheorem{lemma}[subsection]{Lemma}
\newtheorem{corollary}[subsection]{Corollary}
\theoremstyle{definition}
\newtheorem{definition}[subsection]{Definition}
\newtheorem{examples}[subsection]{Examples}
\newtheorem{example}[subsection]{Example}
\newtheorem{remark}[subsection]{Remark}

\def\Id{\mathrm{Id }}

\def\TriDend{\rm{TriDend} }

\def\qGV{\rm{qGV}}

\def\KK{\mathbb{K}}

\def\arbreA{\xymatrix@R=3pt@C=3pt{
&& \\
&*{}\ar@{-}[ur] \ar@{-}[ul] \ar@{-}[d]     &\\
&&
}}
\def\arbreATunTdeux{\xymatrix@R=3pt@C=3pt{
t_1&&t_2 \\
&*{}\ar@{-}[ur] \ar@{-}[ul] \ar@{-}[d]     &\\
&&
}}

\def\arbreAtbis{\xymatrix@R=3pt@C=3pt{
t_1&& \\
&*{}\ar@{-}[ur] \ar@{-}[ul] \ar@{-}[d]     &\\
&&
}}

\def\arbreAtter{\xymatrix@R=3pt@C=3pt{
t_1&\cdots&t_r \\
&*{}\ar@{-}[ur] \ar@{-}[ul] \ar@{-}[d]     &\\
&&
}}

\def\arbreAtterencore{\xymatrix@R=3pt@C=3pt{
t_{r-1}&\cdots&t_r \\
&*{}\ar@{-}[ur] \ar@{-}[ul] \ar@{-}[d]     &\\
&&
}}

\def\arbreAtterdroite{\xymatrix@R=3pt@C=3pt{
&&t_3 \\
&*{}\ar@{-}[ur] \ar@{-}[ul] \ar@{-}[d]     &\\
&&
}}
\def\arbreBA{\xymatrix@R=2pt@C=2pt{
&&&&\\
&&&*{}\ar@{-}[ul] & \\
&&*{}\ar@{-}[uurr] \ar@{-}[uull] \ar@{-}[d]     &&\\
&&&&
}}

\def\arbreBAtbis{\xymatrix@R=2pt@C=2pt{
t_1&&t_2&\cdots&t_r\\
&&&*{}\ar@{-}[ul] & \\
&&*{}\ar@{-}[uurr] \ar@{-}[uull] \ar@{-}[d]     &&\\
&&&&
}}

\def\arbreAB{\xymatrix@R=2pt@C=2pt{
&&&&\\
&*{}\ar@{-}[ur] &&& \\
&&*{}\ar@{-}[uurr] \ar@{-}[uull] \ar@{-}[d]     &&\\
&&&&
}}

\def\arbreABtbis{\xymatrix@R=2pt@C=2pt{
t_1&\cdots&t_{r-1}&&t_r\\
&*{}\ar@{-}[ur] &&& \\
&&*{}\ar@{-}[uurr] \ar@{-}[uull] \ar@{-}[d]     &&\\
&&&&
}}

\def\arbreBB{\xymatrix@R=2pt@C=2pt{
&&*{}&&\\
&&&& \\
&&*{}\ar@{-}[uurr] \ar@{-}[uull] \ar@{-}[d] \ar@{-}[uu]     &&\\
&&&&
}}

\def\arbreABC{\xymatrix@R=1pt@C=1pt{
&&&&&&\\
&*{}\ar@{-}[ur] &&&&& \\
&&*{}\ar@{-}[uurr] &&&&\\
&&&*{}\ar@{-}[uuurrr] \ar@{-}[uuulll] \ar@{-}[d] &&&\\
&&&&&&
}}

\def\arbreAAC{\xymatrix@R=1pt@C=1pt{
&&&&&&\\
&*{}&&&&& \\
&&*{}\ar@{-}[uurr]\ar@{-}[uu]  &&&&\\
&&&*{}\ar@{-}[uuurrr] \ar@{-}[uuulll] \ar@{-}[d] &&&\\
&&&&&&
}}

\def\arbreBAC{\xymatrix@R=1pt@C=1pt{
&&&&&&\\
&&&*{}\ar@{-}[ul] &&& \\
&&*{}\ar@{-}[uurr] &&&&\\
&&&*{}\ar@{-}[uuurrr] \ar@{-}[uuulll] \ar@{-}[d] &&&\\
&&&&&&
}}

\def\arbreACA{\xymatrix@R=1pt@C=1pt{
&&&&&&\\
&*{}\ar@{-}[ur] &&&&*{}\ar@{-}[ul] & \\
&&&&&&\\
&&&*{}\ar@{-}[uuurrr] \ar@{-}[uuulll] \ar@{-}[d] &&&\\
&&&&&&
}}

\def\arbreACC{\xymatrix@R=1pt@C=1pt{
&&&&&&\\
&*{}\ar@{-}[ur]&&&&&\\
&&&&& & \\
&&&*{}\ar@{-}[uuurrr] \ar@{-}[uuulll] \ar@{-}[d] \ar@{-}[uuu]&&&\\
&&&&&&
}}

\def\arbreCAB{\xymatrix@R=1pt@C=1pt{
&&&&&&\\
&&&*{}\ar@{-}[ur] &&& \\
&&&&*{}\ar@{-}[uull] &&\\
&&&*{}\ar@{-}[uuurrr] \ar@{-}[uuulll] \ar@{-}[d] &&&\\
&&&&&&
}}

\def\arbreCBA{\xymatrix@R=1pt@C=1pt{
&&&&&&\\
&&&&&*{}\ar@{-}[ul] & \\
&&&&*{}\ar@{-}[uull] &&\\
&&&*{}\ar@{-}[uuurrr] \ar@{-}[uuulll] \ar@{-}[d] &&&\\
&&&&&&
}}

\def\arbreAt {\xymatrix@R=3pt@C=3pt{
&& \\
&& \\
&*{}\ar@{-}[uur] \ar@{-}[uu]\ar@{-}[uul] \ar@{-}[d]     &\\
&&}}

\def\arbreAttt{\xymatrix@R=3pt@C=3pt{
&&&& \\
&&&&\\
&&*{}\ar@{-}[uurr] \ar@{-}[uur]\ar@{-}[uul] \ar@{-}[uull]\ar@{-}[dd]     &&\\
&&&&\\
&&&&}}

\def\mcorol{\xymatrix@R=3pt@C=3pt{
&&m&& \\
&&\cdots&&\\
&&*{}\ar@{-}[uurr] \ar@{-}[uur]\ar@{-}[uul] \ar@{-}[uull]\ar@{-}[d]     &&\\
&&&&}}

\def\tcorol{\xymatrix@R=3pt@C=3pt{
t_1&t_2&&&t_r \\
&&\cdots&&\\
&&*{}\ar@{-}[uurr] \ar@{-}[uur]\ar@{-}[uul] \ar@{-}[uull]\ar@{-}[d]     &&\\
&&&&}}

\def\tcorolbis{\xymatrix@R=3pt@C=3pt{
t_2&t_3&&&t_r \\
&&\cdots&&\\
&&*{}\ar@{-}[uurr] \ar@{-}[uur]\ar@{-}[uul] \ar@{-}[uull]\ar@{-}[d]     &&\\
&&&&}}

\def\wcorol{\xymatrix@R=3pt@C=3pt{
w_1&w_2&&&w_s \\
&&\cdots&&\\
&&*{}\ar@{-}[uurr] \ar@{-}[uur]\ar@{-}[uul] \ar@{-}[uull]\ar@{-}[d]     &&\\
&&&&}}

\def\tsuccwcorol{\xymatrix@R=3pt@C=3pt{
t_1&t_2&&t_{r-1}&t_r\star w \\
&&\cdots&&\\
&&*{}\ar@{-}[uurr] \ar@{-}[uur]\ar@{-}[uul] \ar@{-}[uull]\ar@{-}[d]     &&\\
&&&&}}

\def\tcdotwcorol{\xymatrix@R=3pt@C=3pt{
t_1&\cdots&t_{r-1}&t_r\star w_1&w_2&\cdots&w_s \\
&&&&&&\\
&&&*{}\ar@{-}[uurrr] \ar@{-}[uur] \ar@{-}[uu] \ar@{-}[uulll]\ar@{-}[uul]\ar@{-}[d]     &&&\\
&&&&&&}}

\def\tprecwcorol{\xymatrix@R=3pt@C=3pt{
t\star w_1&w_2&&&w_s \\
&&\cdots&&\\
&&*{}\ar@{-}[uurr] \ar@{-}[uur]\ar@{-}[uul] \ar@{-}[uull]\ar@{-}[d]     &&\\
&&&&}}
\begin{document}

\author[E. Burgunder, M. Ronco]{Emily Burgunder, Mar\'ia Ronco}
\address{EB: Instytut Matematyczny Polskiej Akademii Nauk,
ul. \'Sniadeckich 8,
P.O. Box 21,
00-956 Warszawa, Poland.}
\email{burgunde@impan.gov.pl}
\urladdr{http://www.impan.pl/~burgunde/}
\address{MOR: Departamento de Matematicas,
  Facultad de Ciencias-Universidad de Valparaiso,
      Avda. Gran Breta{\~n}a 1091,
         Valparaiso, Chile.}
\email{maria.ronco@uv.cl}
\thanks{Our joint work was supported by the Project ECOS-Conicyt C06E01. The first author's work is partially supported by an IPDE grant and an ANR program. The second author's work is partially supported by FONDECYT Project 1085004. Both authors want to thank the Banach Center for its hospitality.}

\title{Tridendriform structure on combinatorial Hopf algebras}
\keywords{Parking functions, bialgebra, tridendriform, planar trees, Solomon-Tits algebra, Poincar\'e-Birkhoff-Witt Cartier-Milnor-Moore theorem}

\begin{abstract} 
We extend the definition of tridendriform bialgebra by introducing a weight $q$.  The subspace of primitive elements of a $q$-tridendriform bialgebra is equipped with an associative product and a natural structure of brace algebra, related by a distributive law. This data is called $q$-Gerstenhaber-Voronov algebras. We prove the equivalence between the categories of connected $q$-tridendriform bialgebras and of $q$-Gerstenhaber-Voronov algebras. The space spanned by surjective maps, as well as the space spanned by parking functions, have natural structures of $q$-tridendriform bialgebras, denoted ${\mbox {\bf ST}(q)}$ and  ${\mbox {\bf PQSym}(q)^*}$, in such a way that ${\mbox {\bf ST}(q)}$ is a sub-tridendriform bialgebra of ${\mbox {\bf PQSym}(q)^*}$. Finally we show that the bialgebra of ${\mathcal M}$-permutations defined by T. Lam and P. Pylyavskyy may be endowed with a natural structure of $q$-tridendriform algebra which is a quotient of ${\mbox {\bf ST}(q)}$.
\end{abstract}

\maketitle
\section*{Introduction} \label{s:int} 

	Some associative algebras admit finer algebraic structures. Dendriform algebras were introduced by J.-L. Loday in \cite{Lod1} as associative algebras whose product splits into two binary operations satisfying some relations.  In particular, any associative product induced somehow by the shuffle product is an example of dendriform structure. The algebraic operad describing dendriform algebras is regular, so it is determined by the free dendriform algebra on one element, which is the algebra of planar binary rooted trees described in \cite{LodRon0}. The natural question which arises is the existence of a regular operad such that the free algebra spanned by one element has as underlying vector space the space spanned by all planar rooted trees. Here are two examples of such an operad:\begin{enumerate}
\item in \cite{Chapoton}, F. Chapoton defined a $K$-algebra as a graded dendriform algebra equipped  with an extra associative product and a boundary map, satisfying certains conditions. When considering the free $K$-algebra on one element, the differential homomorphism on planar trees coincides with the co-boundary map of the associahedra.
\item in a joint work of J.-L. Loday and the second author, see \cite{LodRon}, the authors introduced the notion of tridendriform algebra, which is an associative algebra such that the product splits in three operations.
\end{enumerate}
In fact the free $K$-algebra is the associated graded algebra of the free tridendriform algebra. The strong interest of F. Chapoton\rq s work is that the differential map of the free $K$-algebra spanned by one element gives the co-boundary map of the associahedra, while most of the examples of combinatorial Hopf algebras enter in the tridendendriform case, as the algebras of parking functions and big multi-permutations described in the present work.

In this paper, we define the notion of $q$-tridendriform algebra which is a weighted tridendriform algebra (the weight being on the associative product). The advantage of this notion is that it permits us to deal simultaneously with tridendriform algebras (when $q=1$) and the notion of $K$-algebras (obtained when $q=0$). Mimicking the definition of dendriform bialgebra given in \cite{Ron}, a $q$-tridendriform bialgebra is a bialgebra such that the associative product comes from a $q$-tridendriform structure, which satisfies certain relations with the coproduct.
	
Any dendriform algebra $H$ may be equipped with a brace algebra structure (see \cite{Ron0}), in such a way that whenever $H$ is a dendriform bialgebra the subspace ${\mbox {\it Prim}(H)}$ of primitive elements of $H$ is a sub-brace algebra. Moreover, the category of connected dendriform bialgebras and the category of brace algebras are equivalent (see \cite{Cha} and \cite{Ron}). We extend these results to $q$-tridendriform bialgebras by introducing the notion of $q$-Gerstenhaber-Voronov algebras, denoted ${\mbox {\it GV}_q}$-algebras, which are brace algebras $(B, M_{1n})$ equipped with an associative product $\cdot $ which satisfies the distributive law:
$$\displaylines {
M_{1n}(x\cdot y;z_1,\dots ,z_n)=\hfill\cr
\hfill \sum _{0\leq i\leq j\leq n}q^{j-i}M_{1i}(x;z_1,\dots ,z_i)\cdot z_{i+1}\cdot \dots \cdot z_j\cdot M_{1(n-j)}(y;z_{j+1},\dots ,z_n).\cr}$$
Applying that any $q$-tridendriform bialgebra has a natural structure of dendriform algebra, it is not difficult to see that it is possible to associate to any $q$-tridendriform algebra a ${\mbox {\it GV}_q}$ algebra which has the same underlying vector space. Following the results described in \cite{Ron}, we prove that:\begin{enumerate}
\item  the subspace of primitive elements of a $q$-tridendriform bialgebra $H$ is a sub-${\mbox {\it GV}_q}$ subalgebra of $H$,
\item the free $q$-tridendriform algebra spanned by a vector space $V$ is isomorphic, as a coalgebra, to the cotensor coalgebra of the free ${\mbox {\it GV}_q}$ algebra spanned by $V$,
\item the category of connected $q$-tridendriform bialgebras is equivalent to the category of ${\mbox {\it GV}_q}$-algebras.\end{enumerate}
Our result gives a good-triple of operads for the theory of generalised bialgebra by J.-L. Loday, cf.\cite{Loday GB}. 

Let us point out that, applying F. Chapoton\rq s results, the operad of ${\mbox {\it GV}_0}$-algebras may be equipped with a differential in such a way that we recover the operad ${\mathcal S}_2$ described in \cite{Tur}, also called {\it homotopy G-algebra} in \cite{GerVor}.
\medskip
	
Given a positive integer $n$, let $[n]$ denote the set $\{ 1,\dots ,n\}$. We define a $q$-tridendriform bialgebra structure on the space spanned by all surjective maps from $[n]$ to $[r]$, for all positive integers $r\leq n$, which we denote by ${\mbox {\bf ST}(q)}$. In \cite{NovThi0} and \cite{NovThi1}, J.-C. Novelli and J.-Y. Thibon define the tridendriform bialgebra ${\mbox {\bf PQSym}^*}$ of parking functions, this structure is easily generalized to any $q$. The natural map which associates to any parking function a surjective map is called the standardization, its dual induces a monomorphism of $q$-tridendriform bialgebras from ${\mbox {\bf ST}(q)}$ to ${\mbox {\bf PQSym}^*(q)}$, which differs from the one defined in \cite{NovThi1}. In a forthcoming paper, we apply this homomorphism to prove that ${\mbox {\bf PQSym}^*(q)}$ is free as a tridendriform algebra, as was conjectured in \cite{NovThi1}. On the other hand, we show that the bialgebra ${\mathcal M}{\mbox {\it MR}}$ of big multi-permutations defined by T. Lam and P. Pylyavskyy in \cite{LamPyl} comes from a tridendriform bialgebra structure, and we generalize their construction to any $q$. Finally, we prove that the $q$-tridendriform bialgebra ${\mathcal M}{\mbox {\it MR}(q)}$ is a quotient of ${\mbox {\bf ST}(q)}$.
\medskip

The paper is organized as follows: the first section gives the definition of $q$-tridendriform bialgebra, illustrated by some examples. In the next section we prove the structure theorem for connected $q$-tridendriform bialgebras and ${\mbox {\it GV}_q}$-algebras, which generalises the Cartier-Milnor-Moore Theorem in our context. In the last section we describe the $q$-tridendriform structures of the bialgebras of parking functions and of big multi-permutations and prove that there exists a diagram of $q$-tridendriform bialgebras:
\begin{eqnarray*}
	{\mbox {\bf PQSym}^*(q)} \hookleftarrow {\mbox {\bf ST}(q)} \rightarrow {\mathcal M}{\mbox {\it MR}(q)}
	\end{eqnarray*}
	
\subsection*{Acknowledgment}  We are grateful to Jean-Yves Thibon and Jean-Christophe Novelli for stimulating discussions on parking functions and combinatorial Hopf algebras.  We would like to thank Victor Turchin for helpful explanations on the operad ${\mathcal S}_2$.
 
\bigskip
\section*{Notations}

All vector spaces and algebras we consider are over a field $\KK$. Given a set $X$, we denote by $\KK[X]$ the vector space spanned by $X$. For any vector space $V$, we denote by $V^{\otimes n}$ the tensor product of $V\otimes \dots \otimes V$, $n$ times, over $\KK$. In order to simplify notation, we shall denote an element 
of $V^{\otimes n}$ indistinctly by $x_1\otimes \dots \otimes x_n$ or $(x_1,\dots ,x_n)$.

A \emph{coalgebra} over $\KK$ is a vector space $C$ equipped with a linear homomorphism $\Delta:C\longrightarrow C\otimes C$ which is coassociative. A counit of a coalgebra $(C,\Delta )$ is a linear homomorphism $\epsilon :C\longrightarrow \KK$ such that $\mu\circ (\epsilon \otimes Id_C)\circ \Delta =id_C=\mu\circ(Id_C\otimes \epsilon )\circ \Delta$, where $\mu$ denotes the action of $\KK$ on $C$. The kernel of $\epsilon $ is denoted by ${\overline C}$.

For any coalgebra $(C,\Delta )$ the image of an element $x\in C$ under $\Delta$ is denoted using the Sweedler\rq s notation $\Delta (x)=\sum x_{(1)}\otimes x_{(2)}$. 

Let $(C,\Delta ,\epsilon )$ be a counital coalgebra such that 
$C=\KK\oplus {\overline C}$ , an element $x\in C$ is \emph{primitive} if $\Delta (x)=x\otimes 1_{\KK}+1_{\KK}\otimes x$. The subspace of primitive elements of $C$ is denoted ${\mbox {\it Prim}(C)}$. There exists a natural filtration on ${\overline C}$ given by:\begin{itemize}
\item $F_{1}(C)={\mbox {\it Prim}(C)}$,
\item $F_n(C):=\left\{x \in {\overline C} \mid \bar \Delta(x) \in F_{n-1}C \otimes F_{n-1}C\right\}\ ,$

\noindent where $\bar\Delta (x)=\Delta(x)-1_{\KK}\otimes x-x\otimes 1_{\KK}$.\end{itemize}

\begin{definition}\label{def:connected}
{\rm The counital coalgebra $C$ is said to be \emph{connected} if 
$$C=\KK\oplus  \bigcup_{n \geq 1}F_{n}C\ .$$ }
\end{definition}

Given a vector space $V$, we denote by $T^c(V)$ the space $T(V)=\bigoplus _{n\geq 0}V^{\otimes n}$ equipped with the coalgebra structure given by deconcatenation:
$$\Delta ^c(x_1\otimes \dots \otimes x_n):=\sum _{i=0}^n(x_1\otimes \dots \otimes x_i)\otimes (x_{i+1}\otimes \dots \otimes x_n),$$
for $x_1,\dots ,x_n\in V$.
\bigskip

 Let $n$ be a natural number, the ordered set $\{1,\dots ,n\}$ is denoted $[n]$. If $J=\{j_1,\dots ,j_k\}\subseteq [n]$ and $r\geq 1$, we note by $J+r$ the set $\{ j_1+r,\dots ,j_k+r\}$.
A \emph{composition} of $n$ is an ordered set ${\underline n}=(n_1,\dots ,n_r)$ of positive integers such that $ \sum _{i=1}^rn_i=n$; while a \emph{partition} of $n$ is a sequence of non-negative integers $\lambda = (l_1,\dots ,l_r)$ such that $\sum _{i=1}^rl_i=n$.

The symmetric group of permutations of $n$ elements is denoted by $S_n$. Given a composition ${\underline n}=(n_1,\dots ,n_r)$ of $n$, a ${\underline n}${\it -shuffle} is a permutation $\sigma\in S_n$ such that $\sigma (n_1+\dots +n_i+1)<\dots <\sigma (n_1+\dots +n_{i+1})$, for $0\leq i\leq r-1$. We denote by ${\mbox {\it Sh}(n_1,\dots ,n_r)}$ the set of all 
 ${\underline n}$-shuffles.
 
Consider the set of maps between finite sets. Given $f:[n]\longrightarrow [r]$, we denote $f$ by its image $(f(1),\dots ,f(n))$. 
 The quotient of the set of maps from $[n]$ to $\bigcup _{r\geq 1}[r]$ by the equivalence relation:
$$f\sim h,\ {\rm if,\ and\ only\ if,}\ (f(1),\dots ,f(n))=(h(1),\dots ,h(n)),$$
is denoted ${\mathcal F}_n$.

Given a map $f: [n]\longrightarrow [r]$ and a subset $J=\{i_1<\dots <i_k\}\subseteq [n]$, the restriction of $f$ to $J$ is the map $f\vert_J:=(f(i_1)\dots ,f(i_k))$. Similarly, for a subset $K$ of $[r]$, the co-restriction of $f$ to $K$ is the map $f\vert ^K:=(f(j_1),\dots ,f(j_l))$, where $\{j_1<\dots <j_l\}:=\{ i\in [n]/f(i)\in K\}$.

For any map $f\in {\mathcal F}_n$, let ${\mbox {\it max}(f)}$ be the maximal element in the image of $f$. If $g\in {\mathcal F}_r$ is another map, then $fg$ is the element in ${\mathcal F}_{n+r}$ such that 
$$fg(i):=\begin{cases}f(i),&\ {\rm for}\ 1\leq i\leq n,\\
g(i-n),&\ {\rm for}\ n+1\leq i\leq n+r.\end{cases}$$
We denote by $\cap (f,g)$ the cardinal of the intersection ${\mbox {\it Im}(f)}\cap {\mbox {\it Im}(g)}$.
 \bigskip

\section{Tridendriform bialgebras}
\medskip

Let $q$ be an element of $\KK$, we introduce the definition of $q$-tridendriform algebra in such a way that for $q=1$ we get the definition of tridendriform algebra given in \cite{LodRon}, while for $q=0$ we get the definition of ${\mathcal P}$-algebra described in \cite{Chapoton}.

\begin{definition}\label{def:TriDend} {\rm A \emph{$q$-tridendriform algebra} is a vector space $A$ together with three operations $\prec:A\otimes A\to A$, $\cdot:A\otimes A\to A$ and $\succ:A\otimes A\to A$,
satisfying the following relations:
\begin{enumerate}
\item $(a\prec b)\prec c=a\prec(b\prec c + b\succ c+ q\ b \cdot c)$,
\item $(a\succ b)\prec c=a\succ(b\prec c)$,
\item $(a\prec b+ a\succ b+ q\ a\cdot b)\succ c=a\succ(b\succ c)$,
\item $(a\cdot b)\cdot c=a\cdot(b\cdot c)$,
\item $(a\succ b)\cdot c=a\succ(b\cdot c)$,
\item $(a\prec b)\cdot c=a\cdot(b\succ c)$,
\item $(a\cdot b)\prec c=a\cdot(b\prec c)$.
\end{enumerate} }
\end{definition}

Note that the operation $*:= \ \prec +q\ \cdot +\succ$ is  associative. Moreover, given a $q$-tridendriform algebra $(A, \prec ,\cdot, \succ )$, the space $A$ equipped with the binary operations $\prec $ and ${\overline {\succ}}:=q\cdot + \succ  $ is a dendriform algebra, as defined by J.-L. Loday in \cite{Lod1}.
\medskip

Our main goal is to study the tridendriform algebra structures of the space of parking functions defined in \cite{NovThi1} and of the space of multipermutations introduced in \cite{LamPyl}, which we treat in the next sections. We give in the present section some other examples. The first one is described in \cite{LodRon} for $q=1$ and in \cite{Chapoton} for $q=0$, while the second one is studied in \cite{PalRon} for $q=1$ and in \cite{Chapoton} for $q=0$. 

\begin{examples}\label{extridend} {\bf a)} {\rm Let $T_n$ denote the set of planar rooted trees with $n+1$ leaves. For instance, 
$$T_0=\{ \vert \}\ \ ,\  T_1=\{ \vcenter {\arbreA} \}\ , \ \ T_2=\{\vcenter {\arbreAB} ,\vcenter {\arbreAt} ,\vcenter {\arbreBA}\}\ \ .$$ 
The tree with $n+1$ leaves and a unique vertex (the root) is called the $n$-{\it corolla}, and denoted $c_n$.

Given trees $t^1,\dots ,t^r$, let $\bigvee (t^1,\dots ,t^r)$ be the tree obtained by joining the roots of $t^1,\dots ,t^r$, ordered from left to right, to a new root. It is easy to see that any tree $t\in T_n$ may be written in a unique way as $t=\bigvee (t^1,\dots ,t^r)$, with $t^i\in T_{n_i}$ and $\sum _{i=1}^rn_i+r-1=n$.
On the space $\KK[T_{\infty }]$ spanned by the set $T_{\infty}:=\bigcup _{n\geq 1}T_n$, we define operations $\prec $, $\cdot$ and $\succ $ recursively as follows:
$$\displaylines {
t\succ \vert=t\cdot \vert =\vert \cdot t=\vert \prec t=0,\ {\rm for \ all}\ t\in T_{\infty},\hfill \cr
\vert \succ t=t\prec \vert =t,\ {\rm for \ all}\ t\in T_{\infty},\hfill \cr
t\prec w=\bigvee (t^1,\dots ,t^{r-1},t^r*w),\hfill\cr
t\cdot w:=\bigvee (t^1,\dots ,t^{r-1},t^r*w^1,w^2,\dots ,w^l),\hfill\cr
t\succ w:=\bigvee (t*w^1,w^2,\dots ,w^l),\hfill\cr  }$$
for $t=\bigvee (t^1,\dots ,t^r)$ and $w=\bigvee (w^1,\dots ,w^l)$, where $*$ is the associative product $*=\prec + q\cdot +\succ  $ previously defined.

Note that, even if we need to consider the element $\vert \in T_0$ as the identity for the product $*$ in order to define the tridendriform structure on $\KK[T_{\infty }]$, the elements $\vert \prec \vert $ , $\vert \cdot \vert $ and $\vert 
\prec \vert $ are not defined.

Following  \cite{Chapoton} and \cite{LodRon}, it is immediate to verify that the data $(\KK[T_{\infty }],~\prec ,~\cdot ,\succ )$ is the free $q$-tridendriform algebra spanned by the unique element of $T_1$. 

For any vector space $V$, the $q$-tridendriform structure of $\KK[T_{\infty }]$ extends naturally to the space ${\mbox {\it Tridend}_q}(V):=\bigoplus _{n\geq 1}\KK[T_n]\otimes V^{\otimes n}$ as follows:
$$(t\otimes v_1\otimes \dots \otimes v_n)\circ (w\otimes u_1\otimes \dots \otimes u_m):=(t\circ w)\otimes v_1\otimes \dots \otimes v_n\otimes u_1\otimes \dots \otimes u_m;$$
where $\circ$ is replaced either by $\succ $, or $\prec$, or $\cdot$, respectively. In this case, ${\mbox {\it Tridend}_q}(V)$ is the free $q$-tridendriform algebra spanned by $V$ (see \cite{Chapoton} and \cite{LodRon}).}
\medskip

\noindent {\rm {\bf b)} Let ${\mbox {\bf ST}_n^r}$ be the set of surjective maps from $[n]$ to $[r]$, for $1\leq r\leq n$, and let ${\mbox {\bf ST}_n}:=\bigcup _{r=1}^n {\mbox {\bf ST}_n^r}$. Given $f:[n]\longrightarrow [r]$ there exists a unique surjective map ${\mbox {\it std}(f)}\in {\mbox {\bf ST}_n^r}$ such that $f(i)<f(j)$ if, and only if, ${\mbox {\it std}(f)}(i)<{\mbox {\it std}(f)}(j)$, for $1\leq i,j\leq n$. The map ${\mbox {\it std}(f)}$ is called the standardization of $f$. 

\noindent For example if $f=(2,3,3,5,7)$, then ${\mbox {\it std}(f)}=(1,2,2,3,4)$.

Let $\times:{\mbox {\bf ST}_n^r}\times {\mbox {\bf ST}_m^s}\longrightarrow {\mbox {\bf ST}_{n+m}^{r+s}}$ be the map $$(\alpha , \beta) \mapsto \alpha\times \beta :=(\alpha (1),\dots ,\alpha (n),\beta (1)+r,\dots ,\beta(m)+r).$$

Let ${\mbox {\bf ST}(q)}$ be the vector space ${\mbox {\bf ST}}:=\bigoplus _{n\geq 1}\KK[{\mbox {\bf ST}_n}]$ equipped with the operations $\succ _q$, $\cdot_q$ and $\prec_q$ defined as follows:
$$\displaylines {
f \succ_q g :=\sum _{max(h)<max(k)}q^{\cap(h,k)}hk,\hfill\cr
f \cdot_q g :=\sum _{max(h)=max(k)}q^{\cap(h,k)-1}hk,\hfill\cr
f \prec_q g :=\sum _{max(h)>max(k)}q^{\cap(h,k)}hk,\hfill\cr }$$
where the sums are taken over all pairs of maps $(h,k)$ verifying that $hk$ is surjective, ${\mbox {\it std}(h)}=f $ and ${\mbox {\it std}(k)}=g$, for $f \in {\mbox {\bf ST}_n}$ and $g \in {\mbox {\bf ST}_m}$.
\medskip

For example, if $\alpha = (1,2,1)\in {\mbox {\bf ST}_3}$ and $\beta =(2,1)\in {\mbox {\bf ST}_2}$, then
$$\displaylines {
\alpha \succ \beta =(1,2,1,4,3)+q(1,2,1,3,2)+q(1,2,1,3,1)+(1,3,1,4,2)+(2,3,2,4,1),\hfill\cr
\alpha \cdot \beta =q(1,2,1,2,1)+(1,3,1,3,2)+(2,3,2,3,1),\hfill\cr
\alpha \prec \beta =q(1,3,1,2,1)+(1,4,1,3,2)+q(2,3,2,2,1)+(2,4,2,3,1)+(3,4,3,2,1).\hfill\cr }$$

To check that $({\mbox {\bf ST}(q)}, \succ _q, \cdot_q, \prec_q)$ is a $q$-tridendriform algebra we refer to \cite{Chapoton} and to \cite{PalRon}.}
 \medskip

\noindent {\rm {\bf c)} Let $(A,\cdot )$ be an associative algebra over $\KK$. A {\it Rota-Baxter operator of weight $q$} on $A$ (see \cite{Aguiar}) is a linear map $R:A\rightarrow A$ verifying that:
$$R(x)\cdot R(y)=R(R(x)\cdot y)+R(x\cdot R(y))+qR(x\cdot y),$$
for $x,y\in A$. The data $(A,\cdot ,R)$ is called an associative Rota-Baxter algebra of weight $q$.

Any Rota-Baxter algebra $A$ of weight $q$ has a natural structure of $q$-tridendriform algebra with the associative product $\cdot $ and the operations $\prec $ and $\succ $ given by:
$$\displaylines {
x\prec y:=x\cdot R(y),\cr
x\succ y:=R(x)\cdot y,\cr }$$
for $x,y\in A$.}
\end{examples}
\bigskip

Let $(A,\prec  ,\cdot , \succ )$ be a $q$-tridendriform algebra and let $A_+:=A\oplus \KK$. We denote by $\epsilon:A_+\longrightarrow \KK$ the projection on the second term. For any $x\in A$, we fix $x\succ  1_{\KK}=x\cdot 1_{\KK}=1_{\KK}\cdot x=1_{\KK}\prec  x=0$ and $1_{\KK}\succ  x=x=x\prec  1_{\KK}$.

\begin{definition} {\rm A $q$-tridendriform bialgebra over $\KK$ is a $q$-tridendriform algebra $H$ equipped with a linear homomorphism $\Delta :H_+\longrightarrow H_+\otimes H_+$ verifying the following conditions:\begin{enumerate}
\item $\Delta (1_{\KK})=1_{ \KK}\otimes 1_{\KK},$
\item $(\epsilon \otimes Id)\circ \Delta (x)=1_{\KK}\otimes x$ and $(Id\otimes \epsilon)\circ \Delta (x)=x\otimes 1_{\KK}$, for all $x\in H$,
\item $\Delta (x\succ  y):=\sum (x_{(1)}*y_{(1)})\otimes (x_{(2)}\succ  y_{(2)})$, 
\item $\Delta (x\cdot  y):=\sum (x_{(1)}*y_{(1)})\otimes (x_{(2)}\cdot  y_{(2)})$,
\item $\Delta (x\prec  y):=\sum (x_{(1)}*y_{(1)})\otimes (x_{(2)}\prec  y_{(2)})$,
\end{enumerate}
 where $\Delta (x)=\sum x_{(1)}\otimes x_{(2)}$ for all $x\in H$, and we establish that:\begin{enumerate}
 \item $(x*_qy)\otimes (1_{\KK}\succ _q1_{\KK}):=(x\succ _q y)\otimes 1_{\KK}$,  
 \item $(x*_qy)\otimes (1_{\KK}\cdot _q1_{\KK}):=(x\cdot _q y)\otimes 1_{\KK}$, 
 \item $(x*_qy)\otimes (1_{\KK}\prec _q1_{\KK}):=(x\prec _q y)\otimes 1_{\KK}$, for $x,y\in H$.\end{enumerate}}\end{definition}
\medskip

Note that if $(H,\prec  ,\cdot , \succ ,\Delta)$ is a $q$-tridendriform bialgebra, then $(H_+,*,\Delta )$ is a bialgebra.

We describe the bialgebra structure of the $q$-tridendriform algebras described in Examples {\bf a)} and {\bf b)} of \ref{extridend}.

\noindent {\bf a)} Let $V$ be a vector space. 

\noindent Given elements $x^i=(t^i;v_1^i,\dots ,v_{n_i}^i)\in T_{n_i}\otimes V^{\otimes n_i}$, for $1\leq i\leq r$ and vectors $w_1,\dots ,w_{r-1}\in V$, let $\bigvee _{w_1,\dots ,w_{r-1}}(x^1,\dots ,x^r):=$
$$(\bigvee (t^1,\dots ,t^r);v_1^1,\dots ,v_{n_1}^1,w_1,v_1^2,\dots ,v_{n_{r-1}}^{r-1},w_{r-1},v_1^r,\dots ,v_{n_r}^r),$$
 in $T_{n}\otimes V^{\otimes n}$, where $n={\displaystyle \sum _{i=1}^rn_i+r-1}$.
 
The coproduct $\Delta $ on the free $q$-tridendriform algebra ${\mbox {\it Tridend}_q}(V)$ is the unique linear homomorphism satisfying that:
\begin{enumerate}\item $\Delta (1_{\KK})= 1_{\KK}\otimes 1_{\KK}$.
\item $\Delta (c_n;v_1,\dots ,v_n):=(c_n;v_1,\dots ,v_n)\otimes 1_{\KK} +1_{\KK}\otimes (c_n;v_1,\dots ,v_n)$, for $n\geq 1$.
\item $$\Delta (x):=\sum (x_{(1)}^1*_q\dots *_q x_{(1)}^r)\otimes \bigvee _{w_1,\dots ,w_{r-1}}(x_{(2)}^1,\dots ,x_{(2)}^r)+ x\otimes 1_{\KK},$$
for $x=\bigvee _{w_1,\dots ,w_{r-1}}(x^1,\dots ,x^r)$, with $x^i\in T_{n_i}\otimes V^{\otimes n_i}$.
\end{enumerate}.
\medskip

\noindent {\bf b)} For any $\alpha \in {\mbox {\bf ST}_n}$, define:
$$\Delta (f )=\sum_j f _{(1)}^r\otimes f _{(2)}^r,$$
where the sum is taken over all $0\leq j\leq n$, such that there exists $\delta _r\in {\mbox {\it Sh}(j,n-j)^{-1}}$ with
$f =(f _{(1)}^r\times f _{(2)}^r)\cdot \delta $. 

\noindent For example $\Delta (2,1,3,5,3,4,4,1)=$
$$\displaylines {
1_{\KK}\otimes (2,1,3,5,3,4,4,1)+(1,1)\otimes (1,2,4,2,3,3)+(2,1,1)\otimes (1,3,1,2,2)+\hfill\cr 
\hfill (2,1,3,3,1)\otimes (2,1,1)+(2,1,3,3,4,4,1)\otimes (1)+(2,1,3,5,3,4,4,1)\otimes 1_{\KK}.\cr }$$

The coproduct may also be described in terms of co-restrictions as follows:
$$\Delta (f )=\sum _{j=1}^rf \vert ^{[j]}\otimes {\mbox {\it std}(f \vert ^{[n-j]+j})}.$$

To see that ${\mbox {\bf ST}(q)}$ with $\Delta $ is a $q$-tridendriform bialgebra, suppose that $hk\in {\mbox {\bf ST}_n}$ are such that ${\mbox {\it std}(h)}=f $ and ${\mbox {\it std}(k)}=g$. It is easy to check that:
\begin{enumerate}\item if ${\mbox {\it max}(h)}<{\mbox {\it max}(k)}$, then 
$$\Delta (hk)=\sum_{max(h_{(2)})<max(k_{(2)})} h_{(1)}k_{(1)}\otimes h_{(2)}k_{(2)},$$
\item if ${\mbox {\it max}(h)}={\mbox {\it max}(k)}$, then 
$$\Delta (hk)=\sum_{max(h_{(2)})=max(k_{(2)})} h_{(1)}k_{(1)}\otimes h_{(2)}k_{(2)},$$
\item if ${\mbox {\it max}(h)}>{\mbox {\it max}(k)}$, then 
$$\Delta (hk)=\sum_{max(h_{(2)})>max(k_{(2)})} h_{(1)}k_{(1)}\otimes h_{(2)}k_{(2)},$$
where both $h_{(1)}k_{(1)}$ and $h_{(2)}k_{(2)}$ are surjective.
\end{enumerate}
Moreover, if $h_{(1)}=h\vert ^{[p]\cap Im(h)}$, $h_{(2)}=h\vert ^{[q-p]+p\cap Im(h)}$, $k_{(1)}=k\vert ^{[r]\cap Im(k)}$ and $k_{(2)}=k\vert ^{[s-r]+r\cap Im(k)}$, then $\cap (h,k)=\cap (h_{(1)},k_{(1)})+\cap(h_{(2)},k_{(2)})$.
\bigskip

\section{Structure theorem for tridendriform bialgebras}
\medskip

We want to prove that any connected $q$ tridendriform bialgebra can be reconstructed from the subspace of its primitive elements. In order to do so we need to introduce the notions of brace algebra (see \cite{GerVor}) and of $q$-Gertenhaber-Voronov algebra. Our construction mimics previous results obtained for dendriform bialgebras and brace algebras, whenever the results exposed in the present work are obtained easily applying the methods developed in \cite{Ron} we refer to it for the details of the proofs.

\begin{definition}\label{def:Gerstenhaver-Voronov}
{\rm \begin{enumerate}\item A \emph{brace} algebra is a vector space $B$ equipped with $n+1$-ary operations $M_{1n}:B\otimes B^{\otimes n}\longrightarrow B$, for $n\geq 0$, which satisfy the following conditions:
\begin{enumerate} \item $M_{10}= \Id_B$,
\item $M_{1m}(M_{1n}(x;y_1,\dots ,y_n);z_1,\dots ,z_m)=$
$$\hfill \sum_{0\leq i_1\leq j_1\leq \dots \leq j_n\leq m} M_{1r}(x;z_1,\dots ,z_{i_1},M_{1l_1}(y_1;\dots ,z_{j_1}),\dots ,M_{1l_n}(y_n;\dots , z_{j_n}),\dots ,z_m),$$
for $x,y_1,\dots ,y_n,z_1,\dots ,z_m\in B$, where $l_k=j_k-i_k$, for $1\leq k\leq n$, and $r=\sum _{k=1}^ni_k+m-j_n+n$.\end{enumerate}
\item  A \emph{$q$-Gerstenhaber-Voronov} algebra, ${\mbox {\it GV}_q}$ algebra for short, is 
a vector space $A$ endowed with a brace structure given by operations $M_{1n}$ and an associative product~$\cdot$, 
satisfying the distributive relation:
$$\displaylines {
M_{1n}(x\cdot y;z_1,\dots ,z_n)=\hfill\cr
\hfill\sum _{0\leq i\leq j\leq n}q^{j-i}M_{1i}(x;z_1,\dots ,z_i)\cdot
z_{i+1}\cdot \dots \cdot z_j\cdot M_{1(n-j)}(y;z_{j+1},\dots ,z_n),\cr }$$
for $x,y,z_1,\dots ,z_n\in B$.
\end{enumerate}}
\end{definition}
\medskip

In \cite{Ron} we constructed a functor from the category of dendriform algebras to the category of brace algebras, we recall this construction. Let $(A,\prec , {\tilde {\succ }})$ be a dendriform algebra, we denote: 
\begin{eqnarray*}
\omega _{\prec}(y_1,\dots ,y_i):=y_1\prec (y_2\prec \dots (y_{i-1}\prec y_i))
\\
\omega _{\tilde {\succ}}(y_{i+1},\dots y_n):=((y_{i+1}{\tilde {\succ}} y_{i+2})\ {\tilde {\succ}}\dots )
{\tilde {\succ}} y_n.
\end{eqnarray*}
The brace operations $M_{1n}$ are defined as follows:
$$M_{1n}(x;y_1,\dots ,y_n)=\sum _{i=0}^n (-1)^{n-i}\omega _{\prec}(y_1,\dots ,y_i){\tilde {\succ}}x\prec
\omega _{\tilde {\succ}}(y_{i+1},\dots y_n),\leqno $$
for $n\geq 1$.

Given any $q$-tridendriform algebra $(A,\prec ,\cdot ,\succ)$ we associate to it the brace algebra $(A, M_{1n})$ obtained from the dendriform algebra $(A,\prec ,{\tilde {\succ}}=q\cdot +\succ)$. 

\begin{proposition}\label{oubli}
If $(A,\prec ,\cdot ,\succ)$ is a $q$-tridendriform algebra, then $(A,M_{1n},\cdot )$ is a ${\mbox {\it GV}_q}$ algebra.
\end{proposition}
\begin{proof} We know that $(A,M_{1n})$ is a brace algebra,
therefore it suffices to prove that $\cdot$ and $M_{1n}$ satisfy the distributive relation:
$$\displaylines {
M_{1n}(x\cdot y;z_1,\dots ,z_n)=\hfill\cr
\hfill \sum _{0\leq i\leq j\leq n}q^{j-i}M_{1i}(x;z_1,\dots ,z_i)\cdot
z_{i+1}\cdot \dots \cdot z_j\cdot M_{1(n-j)}(y;z_{j+1},\dots ,z_n),\cr}$$
for $x,y,z_1,\dots ,z_n\in A$.
\medskip

As $\omega_\prec(v_1,\ldots,v_r)\cdot v= v_1\cdot(\omega_\prec(v_2,\ldots,v_r)\succ v)$, for any $v_1,\dots ,v_r,v\in A$, we can split the expression
\begin{eqnarray*}
\sum q^{j-i}M_{1i}(x;z_1,\ldots,z_i)\cdot z_{i+1}\cdot\ldots\cdot z_j\cdot  M_{1(n-j)}(y;z_{j+1},\ldots,z_n)\ ,
\end{eqnarray*}
in three types of terms:
$$\displaylines {
{\bold a)} X_{r,i,j,l}:=(\omega_\prec(z_1,\ldots,z_r)\ \tilde{\succ}\ x\prec \omega_{\tilde{\succ}}(z_{r+1},\ldots,z_i))\cdot z_{i+1}\cdot\ldots\cdot z_j \cdot \hfill\cr
\hfill(\omega_\prec(z_{j+1},\ldots,,z_l)\succ y\prec \omega_{\tilde{\succ}}(z_{l+1},\ldots,,z_n)),\cr }$$
with $j-i\geq 1$
$$\displaylines {
{\bold b)} Y_{r,i,l}:=(\omega_\prec(z_1,\ldots,z_r)\succ x\prec \omega_{\tilde{\succ}}(z_{r+1},\ldots,z_i))\cdot\hfill\cr
\hfill (\omega_\prec(z_{i+1},\ldots,,z_l)\succ y\prec \omega_{\tilde{\succ}}(z_{l+1},\ldots,,z_n)),\cr }$$
$$\displaylines {
{\bold c)} Z_{r,i,l}:=(z_1\cdot(\omega_\prec(z_2,\ldots,z_r)\succ x\prec \omega_{\tilde{\succ}}(z_{r+1},\ldots,z_i)))\cdot\hfill \cr
\hfill (\omega_\prec(z_{i+1},\ldots,,z_l)\succ y\prec \omega_{\tilde{\succ}}(z_{l+1},\ldots,,z_n)).\cr}$$
\medskip

For $j-i\geq 1$, the term $X_{r,i,j,l}$ appears in:
\begin{itemize}
\item $M_{1,i}(x;z_1,\ldots, z_i)\cdot z_{i+1}\cdot\ldots\cdot z_{j}\cdot M_{1, n-j}(y;z_{j+1},\ldots,z_n)$ with the coefficient $q^{j-i}(-1)^{i+l}$
\item $M_{1,i}(x; _1,\ldots,z_i)\cdot z_{i+1}\cdot\ldots\cdot z_{j-1}\cdot M_{1,n-j}(y;z_{j_1},\ldots,z_n)$ with the coefficient $q^{j-i-1}\cdot q\cdot(-1)^{i+l+1}$.
\end{itemize}
So, the coefficient of $X_{r,i,j,l}$ is $q^{j-1}[(-1)^{i+l}+(-1)^{i+l+1}]=0$, and therefore 
$$\displaylines {
\sum q^{j-i}M_{1,i}(x;z_1,\ldots,z_i)\cdot z_{i+1}\cdot\ldots\cdot z_j\cdot  M_{1,n-j}(y;z_{j+1},\ldots,z_n)=\hfill\cr
\hfill \sum _{0\leq r\leq i\leq l\leq n}(-1)^{r+l-i}Y_{r,i,l}+q\sum _{1\leq r\leq i\leq l\leq n}(-1)^{r+l-i}Z_{r,i,l}.\cr }$$

For $r<l$, we have that $Y_{r,i,l}=$
$$\displaylines {
((\omega_\prec(z_1,\ldots,z_r)\succ x)\prec(\omega_{\tilde{\succ}}(z_{r+1},\ldots,z_i)\star\omega_\prec(z_{i+1},\ldots,z_l)))\cdot(y\prec\omega_{\tilde{\succ}}(z_{l+1},\ldots,z_n))=\hfill\cr
\hfill (\omega_\prec(z_1,\ldots, z_r)\succ x)\cdot((\omega_{\tilde{\succ}}(z_{r+1},\ldots,z_i)\star\omega_\prec(z_{i+1},\ldots,,z_l))\succ(y\prec\omega_{\tilde{\succ}}(z_{l+1},\ldots,z_n))).\cr }$$
If $r<i$, then
$$\displaylines {
\omega_{\tilde{\succ}}(z_{r+1},\ldots;z_i)\star\omega_\prec(z_{i+1},\ldots,z_l)=\hfill\cr
\hfill \omega_{\tilde{\succ}}(z_{r+1},\ldots, z_{i+1})\prec\omega_\prec(z_{i+2},\ldots,z_l)+\omega_{\tilde{\succ}}(z_{r+1},\ldots,z_i)\prec\omega_\prec(z_{i+1},\ldots, z_l)'\cr }$$
 which implies that:
$$\displaylines {
\sum_{i=r}^l(-1)^i \omega_{\tilde{\succ}}(z_{r+1},\ldots,z_i)\star\omega_\prec(z_{i+1},\ldots,z_l)=\hfill\cr
\hfill (-1)^r(\omega_\prec(z_{r+1},\ldots, z_l)-z_{r+1}\prec\omega_\prec(z_{r+2},\ldots,z_l))=0,\cr }$$

Therefore, we get that $\sum_{i=r}^l (-1)^{r+l-i} Y_{r,i,l}=0\ ,$ for $r<l$. So,
$$\displaylines {
 \sum _{0\leq r\leq i\leq l\leq n}(-1)^{r+l-i}Y_{r,i,l}=\sum _{0\leq r\leq n}(-1)^rY_{r,r,r}=\hfill \cr
\hfill \sum _{0\leq r\leq n}(-1)^r\omega_\prec(z_1,\ldots,z_r)\succ(x\cdot y)\prec \omega_{\tilde{\succ}}(z_{r+1},\ldots,z_n).\cr }$$

Applying an analogous argument we get that 
$$\displaylines {
\sum _{1\leq r\leq i\leq l\leq n}(-1)^{r+l-i}Z_{r,i,l}=\sum _{1\leq r\leq n}(-1)^rZ_{r,r,r}=\hfill\cr
\hfill (\omega_\prec(z_1,\ldots,z_r)\cdot (x\cdot y)\prec\omega_{\tilde{\succ}}(z_{r+1},\ldots,z_n)).\cr }$$
We can conclude that:
$$\displaylines {
\sum_{0\leq i\leq j\leq n}q^{j-i}M_{1,i}(x;z_1,\ldots,z_i)\cdots z_{i+1}\cdot\ldots\cdot z_j\cdot M_{1,n-j}(y;Z_{j+1},\ldots,z_n)=\hfill\cr
\sum_{r=0}^{n}\omega_\prec(z_1,\ldots,,z_r)\tilde{\succ}(x\cdot y)\prec\omega_{\tilde{\succ}}(z_{r+1},\ldots, z_n)\cr
\hfill M_{1n}(x\cdot y; z_1,\ldots, z_n),\cr }$$
which ends the proof.
\end{proof}

Proposition \ref{oubli} states that there exists a functor ${\mathcal F}$ from the category of $q$-tridendriform algebras to the category of ${\mbox {\it GV}_q}$ algebras. Conversely, for a ${\mbox {\it GV}_q}$-algebra 
$(B,\tilde {M}_{1n},\ \cdot)$, let
\begin{eqnarray*}
U_{\qGV}(B):=\TriDend(B)/{\mathcal I}
\end{eqnarray*} 
where ${\mathcal I}$ is the tridendriform ideal spanned by the elements:
\begin{eqnarray*}
\tilde{M}_{1n}(x;y_1,\dots ,y_n)-\sum _{i=0}^n (-1)^{n-i}\omega _{\prec}(y_1,\dots ,y_i){\tilde {\succ}}x\prec
\omega _{\tilde {\succ}}(y_{i+1},\dots y_n),
\end{eqnarray*}
for all $x,y_1,\dots ,y_n\in B$.
A standard argument shows that $U_{\qGV}$ is a left adjoint of ${\mathcal F}$.
\medskip 

The following result shows that the subspace of primitive elements of $H$ is a ${\mbox {\it GV}_q}$-algebra.
\begin{lemma}\label{lemma:primitive_operation}
Let $(H,\prec,\ \cdot , \succ ,\Delta)$ a $q$-tridendriform bialgebra. If the elements $x,y, z_1,\cdots,z_n$ of $H$ are primitive, then $M_{1n}(x;z_1,\cdots,z_n)$ and $x\cdot y$ are primitive, too.  
\end{lemma}
\begin{proof}
If $x$ and $y$ are primitive, then 
$$\Delta (x\cdot y)= x\cdot y \otimes 1_{\KK} + x\otimes (1_{\KK}\cdot y)+y\otimes (x\cdot 1_{\KK})+1_{\KK}\otimes x\cdot y=x\cdot y \otimes 1_{\KK} +1_{\KK}\otimes x\cdot y,$$
because $1_{\KK}\cdot y=x\cdot 1_{\KK}=0$.

To see that $M_{1n}(x;z_1,\cdots,z_n)$ is primitive, it suffices to note that the brace operation $M_{1n}$ on the $q$-tridendriform algebra $(H,\prec,\ \cdot , \succ )$ coincides with the brace defined on the dendriform algebra 
$(H,\prec,\ {\tilde {\succ}}:=q\cdot + \succ )$ in \cite{Ron}. Since $(H,\prec,{\tilde {\succ}} ,\Delta )$ is a dendriform bialgebra, it suffices to apply the result of \cite{Ron}.
\end{proof}
\bigskip

Let $(H, \prec, \cdot ,\succ ,\Delta)$ be a $q$-tridendriform bialgebra, we say that $H$ is connected if $(H_+,\Delta )$ is a connected coalgebra. For $n\geq 1$, define linear maps $\succ ^n:H^{\otimes n}\longrightarrow H$ and ${\overline {\Delta }}^n:H\longrightarrow H^{\otimes n}$ as follows:
\begin{eqnarray} \succ^1&=&\Id\\
\succ^n&=&\succ^{n-1}\circ(\Id^{\otimes n-2}\otimes\succ)\\
{\overline {\Delta}}^1&=&\Id\\
{\overline {\Delta}}^n&=&(\Id^{\otimes n-2}\otimes{\overline {\Delta}})\circ{\overline {\Delta}}^{n-1},\end{eqnarray}

Note that in this case $H={\overline {H_+}}$.  

Let $e_{tri}:H\longrightarrow H$ be the linear map given by $$e(x):=\sum_{n\geq 1}(-1)^{n+1}\succ^n\circ\bar\Delta^n(x).$$

For any element $x\in H$, we have that $e_{tri}(x)=x-\sum x_{(1)}\succ e_{tri}(x_{(2)})$, for ${\overline {\Delta }}(x)=\sum x_{(1)}\otimes x_{(2)}$. 
The previous equality implies that:
\begin{enumerate}\item If $x\in {\mbox {\it Prim}(H)}$, then $e_{tri}(x)=x$.
\item Whenever $x=y\succ z\in F_n(H)$ for elements $y,z\in F_r(H)$ with $r<n$, a recursive argument on $n$ shows that $e_{tri}(x)=0$.\end{enumerate}
So, we may consider $e_{tri}$ as a projection from $H$ to ${\mbox {\it Prim}(H)}$. Moreover, the Proposition below shows that any element $x\in H$ may be described in terms of the operation $\succ $ and primitive elements.

\begin{proposition}\label{structurefortridend} Let $(H, \prec, \cdot ,\succ ,\Delta)$ be a connected $q$-tridendriform bialgebra. Any element $x\in F_n(H)$ satisfies that:
$$\displaylines {
x=e_{tri}(x)+\sum e_{tri}(x_{(1)})\succ e_{tri}(x_{(2)}) +\dots +\sum \omega _{{\succ}}(e_{tri}(x_{(1)}),\dots ,e_{tri}(x_{(n)}))=\cr
\hfill \sum _{r=1}^n(\sum \omega _{{\succ}}(e_{tri}(x_{(1)}),\dots ,e_{tri}(x_{(r)}))),\cr }$$
where ${\overline {\Delta}}^r(x)=\sum x_{(1)}\otimes \dots \otimes x_{(r)}$ and 

\noindent $ \omega _{{\succ}}(e_{tri}(x_{(1)}),\dots ,e_{tri}(x_{(r)})):=(((e_{tri}(x_{(1)})\succ e_{tri}(x_{(2)}))\succ e_{tri}(x_{(3)}))\succ \dots )\succ e_{tri}(x_{(r)}).$
\end{proposition} 

\begin{proof} Since $H$ is connected, an element $x\in F_n(H)$, for some $n \geq 1$. We have also that $x-e_{tri}(x)=\sum x_{(1)}\succ e_{tri}(x_{(2)})$. The result is clear for $n=1$.

For $n\geq 2$, ${\overline {\Delta }}(x)=\sum x_{(1)}\otimes x_{(2)}$, with $x_{(1)}$ and $x_{(2)}$ in $F_{n-1}(T)$. By a recursive argument, we get that
$$x_{(1)}=\sum _{r=1}^{n-1}(\sum \omega _{{\succ}}(e_{tri}(x_{(1)(1)}),\dots ,e_{tri}(x_{(1)(r)}))).$$
So, $$\displaylines {
x=e_{tri}(x)+\sum \bigl(\sum _{r=1}^{n-1}(\sum \omega _{{\succ}}(e_{tri}(x_{(1)(1)}),\dots ,e_{tri}(x_{(1)(r)}))\succ e_{tri}(x_{(2)}))\bigr)=\hfill\cr
\hfill \sum _{r=1}^n(\sum \omega _{{\succ}}(e_{tri}(x_{(1)}),\dots ,e_{tri}(x_{(r)}))),\cr }$$
which ends the proof.\end{proof}

\begin{remark}\label{omegas}{\rm (see \cite{Ron}) If the elements $x_1,\dots ,x_n$ belong to ${\mbox {\it Prim}(H)}$, then 
$$\Delta (\omega _{{\succ}}(x_1,\dots ,x_n))=\sum _{i=0}^n\omega _{{\succ}}(x_1,\dots ,x_i)\otimes \omega _{{\succ}}(x_{i+1},\dots ,x_n),$$
where $\omega _{{\succ}}(\emptyset):=1_{\KK}$.}\end{remark}

Note that Proposition \ref{structurefortridend} and Remark \ref{omegas} imply that for any connected $q$-tridendriform bialgebra $(H, \prec, \cdot ,\succ ,\Delta)$, the linear homomorphism from $(H_+,\Delta )$ to the cotensor coalgebra 
$T^c({\mbox {\it Prim}(H)})$ which sends an element $x\in F_n(H)$ to $\sum _{r=1}^n(\sum e_{tri}(x_{(1)(1)})\otimes \dots \otimes e_{tri}(x_{(1)(r)}))$ is an isomorphism of coalgebras.
\medskip

We have proved that the subspace of primitive elements of a $q$-tridendriform bialgebra has a natural structure of ${\mbox {\it GV}_q}$ algebra. In fact, there exists an equivalence between the category of connected $q$-tridendriform bialgebras and the category of ${\mbox {\it GV}_q}$ algebras. The last part of the section is devoted to this result.

\begin{proposition}
Let $V$ be a $\KK$-vector space. The primitive part of the free $q$-tridendriform algebra ${\mbox {\it Tridend}_q}(V)$ is the free ${\mbox {\it GV}_q}$ algebra over $V$.
\end{proposition}
\begin{proof}
To prove the result we may assume that $V$ is a finite dimensional space over $\KK$, the general case follows by taking a direct limit. 

Suppose that ${\mbox {\it dim}_K(V)}=m$ and that ${\mathcal B}$ is a basis of $V$. We know that a basis for the space ${\mbox {\it Tridend}_q}(V)_n$, of homogeneous elements of degree $n$ of ${\mbox {\it Tridend}_q}(V)$, is given by the set $T_n\times {\mathcal B}^n$ whose cardinal is $C_nm^n$, where $C_n=\vert T_n\vert$ is the super-Catalan number.  But the vector space  ${\mbox {\it Tridend}_q}(V)$ is isomorphic to the tensor space $T({\mbox {\it Prim}}({\mbox {\it Tridend}_q}(V))$, which implies that the dimension of ${\mbox {\it Prim}}({\mbox {\it Tridend}_q}(V)_n)$ is $C_{n-1}m^{n}$. 

The paragraph above implies that there exists a bijection between the set $T_{n-1}\times {\mathcal B}^n$ of elements $(t;b_1,\dots ,b_n)\in T_n\times {\mathcal B}^n$ such that $t=\bigvee (\vert ,t^2,\dots ,t^r)$ and a basis of the space of primitive elements of ${\mbox {\it Tridend}}(V)$. Let $T_n^{\succ }$ denote the set of all trees in $T_n$ of the form $\bigvee (t^1,\dots ,t^r)$ with $\vert t^1\vert \geq 1$. We have that for any $t=\bigvee (\vert ,t^2,\dots ,t^r)$,
$e_{tri}((t;b_1,\dots ,b_n)=(t;b_1,\dots ,b_n) + z$ where $z$ belongs to the subspace spanned by $T_n^{\succ }\times {\mathcal B}^n$, which implies that the set of elements $e_{tri}((t;b_1,\dots ,b_n)$, with 
$t\in T_n^{\succ }$ form a basis of ${\mbox {\it Prim}}({\mbox {\it Tridend}_q}(V))$. 

On the other hand, the free ${\mbox {\it GV}_q}$ algebra ${\mbox {\it GV}_q}(V)$ spanned by $V$ has a basis ${\mbox {\it GV}_q}({\mathcal B})$ whose elements of degree $n$ may be described recursively as follows:
\begin{enumerate}\item ${\mbox {\it GV}_q}({\mathcal B})_1={\mathcal B},$
\item ${\mbox {\it GV}_q}({\mathcal B})_n$ is the set of all elements of the form $$M_{1n_1}(b_1;y_1^1,\dots ,y_{n_1}^1)\cdot \dots \cdot M_{1n_r}(b_r;y_1^r,\dots ,y_{n_r}^r),$$ where $b_1,\dots ,b_r\in {\mathcal B}$, $y_j^i\in 
{\mbox {\it GV}_q}({\mathcal B})_{n_{ij}}$, with $n_{ij}<n$, and $0\leq n_i$ for $1\leq i\leq r$.
\end{enumerate}
To end the proof it suffices to note that there exists a unique bijective map $\varphi$ from ${\mbox {\it GV}_q}({\mathcal B})$ to $T_n^{\succ }\times {\mathcal B}^n$ such that:\begin{enumerate}
\item $\varphi _1(b)=(c _1 ,b)$, for $b\in {\mathcal B}$,
\item $\varphi _m(M_{1n}(b;y_1,\dots ,y_n)=(c_1,b)\prec \omega _{\succ }(\varphi _{m_1}(y_1),\dots ,\varphi _{m_n}(y_n)),$
where $\omega _{\succ}(x_1,\dots ,x_n)=(((x_1\succ x_2)\succ x_3)\dots )\succ x_n$.
\item $\varphi _m(y_1,\dots ,y_n)=\varphi _{m_1}(y_1)\cdot \dots \cdot \varphi _{m_n}(y_n)$.\end{enumerate}
\medskip

Since ${\mbox {\it Prim}}({\mbox {\it Tridend}_q}(V))$ is a ${\mbox {\it GV}_q}$ algebra which contains $V$, it must be isomorphic to ${\mbox {\it GV}_q}(V)$.\end{proof}
\medskip

Applying the previous results we may show that the category of connected $q$-tridendriform bialgebras is equivalent to the category of $q$-Gertenhaber-Voronov algebras. 

\begin{theorem} Let $(H,\prec ,\cdot ,\succ )$ be a $q$-tridendriform bialgebra.\begin{enumerate}
\item If $H$ is connected then $H$ is isomorphic to the enveloping tridendriform algebra $U_{\qGV}({\mbox {\it Prim}}(H))$.
\item Any ${\mbox {\it GV}_q}$ algebra $B$ is isomorphic to the primitive algebra ${\mbox {\it Prim}}(U_{\qGV}(B))$ of its enveloping algebra.\end{enumerate}
\end{theorem}

\begin{proof} We give the main line of the proof, for the details we refer to the analogous result for connected dendriform bialgebras proved in \cite{Ron}.

If $H$ is a connected $q$-tridendriform bialgebra, we know that $H$ is isomorphic as a coalgebra to $T^c({\mbox {\it Prim}}(H))$. To prove the first statement, it suffices to verify that the composition:
$$H\longrightarrow T^c({\mbox {\it Prim}}(H))\cong {\mbox {\it Tridend}_q}({\mbox {\it Prim}}(H))\longrightarrow U_{\qGV}({\mbox {\it Prim}}(H)),$$
is an isomorphism of $q$-tridendriform bialgebras, which is straightforward to check.
\medskip

For the second point, it is clear that $B\subseteq {\mbox {\it Prim}}(U_{\qGV}(B))$. On the other hand, we have that ${\mbox {\it Prim}}({\mbox {\it Tridend}_q}(B))={\mbox {\it GV}_q}(B)$. Since in enveloping algebra $U_{\qGV}(B)$ we identify the elements of ${\mbox {\it Prim}}({\mbox {\it Tridend}_q}(B))$ with elements of $B$, we get the result.\end{proof}

\section{Tridendriform structure on the spaces of parking functions and of multipermutations }

\subsection{Parking functions }
\medskip

In \cite{NovThi1}, J.-C. Novelli and J.-Y. Thibon defined a $1$-tridendriform structure on the space ${\mbox {\bf PQSym}^*}$ spanned by parking functions. We show that their result extends naturally to any $q\in \KK$, in such a way that the coalgebra structure on the parking functions gives a $q$-tridendriform bialgebra on ${\mbox {\bf PQSym}^*}$, for all $q$. Our main result is that the $q$-tridendriform bialgebra ${\mbox {\bf ST}(q)}$ is a sub-tridendriform bialgebra of ${\mbox {\bf PQSym}^*(q)}$. We begin by recalling some basic definitions about parking functions, for a more complete description we refer to \cite{NovThi1}.

\begin{definition}\label{def:Parking}
{\rm  A map $f: [n]\to [n]$ is called a \emph{$n$-non-decreasing parking function} if $f(i)\leq i $ for $1\leq i\leq n$. 
The set of $n$-non-decreasing parking functions is denoted by ${\mbox {\it NDPF}_n}$.

The composition $f:=f^{\uparrow}\circ \sigma$ of a non decreasing parking function 

\noindent $f^{\uparrow}\in NDPF_n$ and a permutation  $\sigma \in S_n$ is called a \emph{$n$-parking function}. The set of $n$-parking functions is denoted by ${\mbox {\it PF}}_n$.}
\end{definition}

Note that given a parking function $f=f^{\uparrow}\circ \sigma$, the non-decreasing parking function $f^{\uparrow}$ is uniquely determined but $\sigma$ is not unique. However, if
$r_i=\vert f^{-1}(i)\vert $, for $1\leq i\leq n$, then there exists a unique $(r_1,\dots ,r_n)$-shuffle $\sigma _0$ such that $f=f^{\uparrow}\circ \sigma_0^{-1}$

\begin{example}
In low dimensions, the sets ${\mbox {\it NDPF}_n}$ and ${\mbox {\it PF}_n}$ are described as follows:
\begin{enumerate}[$\bullet$]
\item ${\mbox {\it NDPF}_1}=\{(1)\}$, ${\mbox {\it NDPF}}_2=\{(1,2), (1,1)\}$,
 \item ${\mbox {\it NDPF}}_3=\{ (1,2,3), (1,1,2),(1,1,3),(1,2,2), (1,1,1)\}$,
\item ${\mbox {\it PF}_1}=\{ (1)\}$, ${\mbox {\it PF}_2}=\{ (1,2),(1,1), (2,1)\}$, 
\item ${\mbox {\it PF}_3}=
S_3\bigcup  (1,2,2)\circ {\mbox {\it Sh}(1,2)}^{-1}\bigcup (1, 1,2)\circ {\mbox {\it Sh}(2,1)}^{-1}
\bigcup (1, 1,3)\circ {\mbox {\it Sh}(2,1)}^{-1}\bigcup \{(1,1,1)\}$.
\end{enumerate}
\end{example}
\medskip

Recall that the cardinal of ${\mbox {\it NDPF}_n}$ is the Catalan number $c_n=\frac{(2n)!} {(n+1)! n!}$, while the number of elements of ${\mbox {\it PF}_n}$ is $(n+1)^{n-1}$.

 The map ${\mbox {\it Park}}:{\displaystyle \bigcup _{n\geq 1}}{\mathcal F}_{n}\longrightarrow {\displaystyle \bigcup _{n\geq 1}}{\mbox {\it PF}_n}$ (see \cite{NovThi0}) is defined as follows.
Let $f^{\uparrow }:[n]\longrightarrow [r]$ be a non-decreasing function, the element ${\mbox {\it Park}(f^{\uparrow })}$ is given by:
$${\mbox {\it Park}(f^{\uparrow})}(j):=\begin{cases}1,&\ {\rm for}\ j=1,\\
Min\{ {\mbox {\it Park}(f^{\uparrow})}(j-1))+f^{\uparrow}(j)-f^{\uparrow }(j-1), \ j\}, & \ {\rm for}\ j>1.\end{cases}$$

Suppose now that $f=f^{\uparrow}\circ \sigma$, where $f^{\uparrow }$ is a non-decreasing function and $\sigma $ is a permutation. Define
$${\mbox {\it Park(f)}}:={\mbox {\it Park}(f^{\uparrow})}\circ \sigma.$$ 

\subsubsection {\rm {\bf Remark.}}  {\rm Let $f\in {\mbox {\it PF}_n}$ be a parking function. It is easy to check that:}\begin{enumerate}
\item $f(i)=f(j)$ {\rm if, and only if} ${\mbox {\it Park}(f)}(i)={\mbox {\it Park}(f)(j)}$,
\item $f(i)<f(j)$ {\rm  if, and only if } ${\mbox {\it Park}(f)}(i)<{\mbox {\it Park}(f)}(j),$
\end{enumerate}
for $1\leq i,j\leq n$.

\medskip

There exists a natural embedding $\times_{\mathcal P}:{\mbox {\it PF}}_n\times {\mbox {\it PF}}_m\hookrightarrow {\mbox {\it PF}}_{n+m}$ given by:
$$f\times_{\mathcal P} g:=(f(1),\dots ,f(n),g(1)+n,\dots ,g(m)+n),\ {\rm for}\ f\in {\mbox {\it PF}}_n\  {\rm and}\ g\in {\mbox {\it PF}}_m.$$
Note that it is not the same that the one considered on ${\mbox {\bf ST}}$, which is denoted $\times$.
\bigskip

Let ${\mbox {\bf PQSym}^*}$ denote the vector space spanned by the set $ \bigcup _{n\geq 1}{\mbox {\it PF}_n}$ of parking functions. For any $q\in \KK$, we endow ${\mbox {\bf PQSym}^*}$ with a structure of $q$-tridendriform bialgebra, which extends the J.-C. Novelli and J.-Y. Thibon  construction of $1$-tridendriform bialgebra on this space.

The binary operations $\prec _q$, $\cdot _q$ and $\succ _q$ on ${\mbox {\bf PQSym}^*}$ are defined in a similar way that in the case of ${\mbox {\bf ST}}$:
$$\displaylines {
f \prec_q g :=\sum _{max(h)>max(k)}q^{\cap(h,k)}hk,\hfill\cr
f \cdot_q g :=\sum _{max(h)=max(k)}q^{\cap(h,k)-1}hk,\hfill\cr
f \succ_q g :=\sum _{max(h)<max(k)}q^{\cap(h,k)}hk,\hfill\cr }$$
where the sums are taken over all pairs of maps $(h,k)$ verifying that $hk$ is parking, ${\mbox {\it Park}(h)}=f $ and ${\mbox {\it Park}(k)}=g$, for $f,g \in \bigcup _{n\geq 1}{\mbox {\it PF}_n}$.

For example, if $f = (1,3,1)\in {\mbox {\it PF}_3}$ and $g =(1,1)\in {\mbox {\it PF}_1}$, then
\begin{eqnarray*}
f \prec_q g &=&(2,4,2,1,1)+(2,5,2,1,1)+(3,5,3,1,1)+ q((1,3,1,1,1)+ (1,4,1,1,1)+\\
&&(1,5,1,1,1))+(1,3,1,2,2)+ (1,4,1,2,2)+(1,4,1,3,3)+(1,5,1,2,2)+\\
&&(1,5,1,3,3)+(1,5,2,4,4),\\
f \succ_q g &=&(1,3,1,4,4),\\
f \cdot_q g &=&(1,3,1,3,3),\end{eqnarray*}

Applying the same arguments that in \cite{NovThi1} it is easily seen that $({\mbox {\bf PQSym}^*},~\prec _q,~\ \cdot_q,\succ_q)$ is a $q$-tridendriform algebra. We denote by ${\mbox {\bf PQSym}^*(q)}$ the space 
${\mbox {\bf PQSym}^*}$ endowed with the structure of $q$-tridendriform algebra.

Define a coproduct $\Delta $ on ${\mbox {\bf PQSym}^*}$ by setting for $f\in {\mbox {\it PF}_n}$ :
$$\Delta (f)=\sum _jf_{(1)}^j\otimes f_{(2)}^j,$$
where the sum is taken over all $0\leq j\leq n$ such that there exist $f_{(1)}^j\in {\mbox {\it PF}_j}$, $f_{(2)}^j\in {\mbox {\it PF}_{n-j}}$ and $\delta _j\in {\mbox {\it Sh}(j,n-j)^{-1}}$ with $f=(f_{(1)}^j\times_{\mathcal P} f_{(2)}^j)\circ \delta _j$.
Note that for any $0\leq j\leq n$, if the decomposition $f=(f_{(1)}^j\times_{\mathcal P} f_{(2)}^j)\circ \delta _j$ exists, then the elements $f_{(1)}^j$, $f_{(2)}^j$ and $\delta _j$ are unique.

For example, $$\displaylines {
\Delta ((1,5,5,3,6,2,3))= (1,5,5,3,6,2,3)\otimes 1_{\KK} + (1,3,2,3)\otimes (1,1,2) +\hfill\cr
\hfill (1,2)\otimes (3,3,1,4,1)+(1)\otimes (4,4,2,5,1,2)+1_{\KK}\otimes (1,5,5,3,6,2,3).\cr }$$ 

\begin{proposition} The $q$-tridendriform algebra ${\mbox {\bf PQSym}^*(q)}$, equipped with $\Delta$ is a $q$-tridendriform bialgebra.
\end{proposition}

\begin{proof} Let us see that 
$$\Delta (f\succ _q g)=\sum (f_{(1)}*_qg_{(1)})\otimes (f_{(2)}\succ _qg_{(2)}),$$
for $f\in {\mbox {\it PF}_n}$ and $g\in {\mbox {\it PF}_m}$. The other relations may be verified in a similar way. 

Let $h\in {\mathcal F}_n$ and $k\in {\mathcal F}_m$ be such that $hk\in {\mbox {\it PF}_{n+m}}$, ${\mbox {\it Park}(h)}=f $, ${\mbox {\it Park}(k)}=g$ and ${\mbox {\it max}(h)}<{\mbox {\it max}(k)}$. Suppose that for $0\leq j\leq n+m$, the function $hk$ may be written as:
$$hk=((hk)_{(1)}^j\times_{\mathcal P} (hk)_{(2)}^j)\circ \delta _j,$$
with $(hk)_{(1)}^j\in {\mbox {\it PF}_{j}}$, $(hk)_{(2)}^j\in {\mbox {\it PF}_{n+m-j}}$ and $\delta _j\in {\mbox {\it Sh}(j,n+m-j)^{-1}}.$

Then there exists a unique integer $0\leq r\leq j$ such that $(hk)_{(1)}^j=h_{(1)}^rk_{(1)}^{j-r}$, $(hk)_{(2)}^j=h_{(2)}^rk_{(2)}^{j-r}$, and $\delta _j=(\delta _j^1\times \delta _j^2)\cdot \gamma$, with $\delta _j^1\in
{\mbox {\it Sh}(r,n-r)^{-1}}$, $\delta _j^2\in {\mbox {\it Sh}(j-r,m+r-j)^{-1}}$ and $\gamma \in {\mbox {\it Sh}(n,m)^{-1}}$. 
In this case we have that $f=({\mbox {\it Park}(h_{(1)}^r)}\times_{\mathcal P} {\mbox {\it Park }(h_{(2)}^r)})\circ\delta _j^1$ and $g=({\mbox {\it Park}(k_{(1)}^{j-r})}\times_{\mathcal P} {\mbox {\it Park }(k_{(2)}^{j-r})})\circ\delta _j^2$.
Finally, it is easy to see that ${\mbox {\it max}(h_{(2)}^r)}<{\mbox {\it max}(k_{(2)}^{j-r})}$.
So, to any term in $\Delta (f\succ _q g)$ corresponds a term in $(*\times \succ _q)\circ (\Delta \times \Delta )(f\otimes g)$.
\medskip
 
 Conversely, suppose that $f=(f_{(1)}^r\times_{\mathcal P} f_{(2)}^r)\circ \delta _r$ and $g=(g_{(1)}^l\times_{\mathcal P} g_{(2)}^l)\circ \gamma _l$, for parking functions $f_{(1)}^r,\ f_{(2)}^r,\  g_{(1)}^l,\ g_{(2)}^l$ and permutations $\delta _r\in {\mbox {\it Sh}(r,n-r)^{-1}}$ and $\gamma _l\in {\mbox {\it Sh}(l,m-l)^{-1}}$. 
 Let $h_1\in {\mathcal F}_r$, $h_2\in {\mathcal F}_{n-r}$, $k_1\in {\mathcal F}_l$ and $k_2\in {\mathcal F}_{m-l}$ be such that:\begin{enumerate}
 \item $h_1k_1\in {\mbox {\it PF}_{r+l}}$ and $h_2k_2\in {\mbox {\it PF}_{n+m-r-l}}$,
 \item ${\mbox {\it Park}(h_i)}=f_{(i)}^r$ and ${\mbox {\it Park}(k_i)}=g_{(i)}^l$, for $i=1,2$,
 \item ${\mbox {\it max}(h_{(2)})}<{\mbox {\it max}(k_{(2)})}$.\end{enumerate}
 The elements $h=(h_1\times_{\mathcal P} h_2)\circ \delta _r\in {\mathcal F}_n$ and $k=(k_1\times_{\mathcal P} k_2)\circ \gamma _l\in {\mathcal F}_m$ verify that $hk\in {\mbox {\it PF}_{n+m}}$, ${\mbox {\it Park}(h)}=f$, ${\mbox {\it Park}(k)}=g$ and 
 ${\mbox {\it max}(h)}<{\mbox {\it max}(k)}$.
\end{proof}

Note that any surjective map from $\{1,\dots ,n\}$ to $\{ 1,\dots ,r\}$ is a parking function. There exist a natural map from ${\mbox {\it PF}_n}$ to ${\mbox {\bf ST}_n}$ given by $f\mapsto {\mbox {\it std}(f)}$ which is surjective but not injective, and coincides with the identity map on ${\mbox {\bf ST}_n}$. The linear map $\alpha _n: \KK[{\mbox {\bf ST}_n}]\longrightarrow \KK[{\mbox {\it PF}_n}]$ given by 
$$\alpha _n(f)=\sum _{h\in {\mbox {\it PF}_n}\vert {\mbox {\it std}(h)}=f}h,$$ is a monomorphism, for $n\geq 1$.

\begin{theorem} The bialgebra ${\mbox {\bf ST}(q)}$ is a sub-$q$-tridendriform bialgebra of ${\mbox {\bf PQSym}^*(q)}$.\end{theorem} 

\begin{proof} Let $f\in {\mbox {\bf ST}_n}$ and $g\in {\mbox {\bf ST}_m}$. 
Given $u\in {\mbox {\it PF}_{n+m}}$ there exist unique functions $u_1\in {\mathcal F}_n$ and $u_2\in {\mathcal F}_m$ such that $u=u_1u_2$, and unique functions $h\in {\mathcal F}_n$ and $k\in  {\mathcal F}_m$ such 
that ${\mbox {\it std}(u)}=hk$. Moreover, we have that ${\mbox {\it std}(u_1)}={\mbox {\it std}(h)}$ and ${\mbox {\it std}(u_2)}={\mbox {\it std}(k)}$.

Note that $$\alpha _{n+m}(f\succ _q g)=\sum _{u\in {\mbox {\it PF}_{n+m}}}q^{\cap(h,k)}u,$$
where the sum is extended over all the functions $u$ such that ${\mbox {\it std}(u)}=hk$, with ${\mbox {\it std}(h)}=f$, ${\mbox {\it std}(k)}=g$ and ${\mbox {\it max}(h)}<{\mbox {\it max}(k)}$.

On the other hand, $$\alpha _n(f)\succ _q \alpha _m(g)=\sum _{u\in {\mbox {\it PF}_{n+m}}}q^{\cap(u_1,u_2)}u,$$
where the sum is extended over all the functions $u=u_1u_2$ such that ${\mbox {\it std}({\mbox {\it Park }(u_1)})}=f$, ${\mbox {\it std}({\mbox {\it Park }(u_2)})}=g$ and ${\mbox {\it max}(u_1)}<{\mbox {\it max}(u_2)}$. 

It is immediate to check that:\begin{enumerate}
\item ${\mbox {\it std}({\mbox {\it Park }(u_i)})}={\mbox {\it std}(u_i)}$, for $i=1,2$,
\item if $u=u_1u_2$ and ${\mbox {\it std}(u)}=hk$, then ${\cap(h,k)}={\cap(u_1,u_2)}$,
\item if ${\mbox {\it std}(u_1)}=f$ and ${\mbox {\it std}(u_2)}=g$, then ${\mbox {\it std}(u_1u_2)}=hk$ with ${\mbox {\it std}(h)}=f$ and ${\mbox {\it std}(k)}=g$,
\item if ${\mbox {\it std}(u_1u_2)}=hk$, then ${\mbox {\it max}(h)}<{\mbox {\it max}(k)}$ if, and only if, ${\mbox {\it max}(u_1)}<{\mbox {\it max}(u_2)}$.
\end{enumerate}
We may conclude that $\alpha _{n+m}(f\succ _q g)=\alpha _n(f)\succ _q \alpha _m(g)$. 

 Similar arguments show that $\alpha _{n+m}(f\cdot _q g)=\alpha _n(f)\cdot _q \alpha _m(g)$ and $\alpha _{n+m}(f\prec_q g)=
\alpha _n(f)\prec _q \alpha _m(g)$.

So, ${\mbox {\bf ST}}$ is a $q$-tridendriform subalgebra of ${\mbox {\bf PQSym}^*}$.
\medskip

To prove that $\alpha $ is a coalgebra homomorphism, suppose that $h\in {\mbox {\it PF}_n}$ and $0\leq r\leq n$ are such that ${\mbox {\it std}(h)}=f$ and 
$$h=(h_{(1)}^r\times_{\mathcal P} h_{(2)}^r)\circ \delta _r,\ {\rm for}\ h_{(1)}^r\in {\mbox {\it PF}_r},\ h_{(2)}^r\in {\mbox {\it PF}_{n-r}}\ {\rm and}\ \delta _r\in {\mbox {\it Sh}(r,n-r)^{-1}}.$$
Let $f_{(1)}^r:={\mbox {\it std}(h_{(1)}^r)}$ and $f_{(2)}^r:={\mbox {\it std}(h_{(2)}^r)}$, we get that $f= (f_{(1)}^r\times f_{(2)}^r)\circ \delta _r$.
\medskip

Conversely, suppose that $f= (f_{(1)}^r\times f_{(2)}^r)\circ \delta _r$, for some $f_{(1)}^r\in {\mbox {\bf ST}_r}$, $f_{(2)}^r\in {\mbox {\bf ST}_{n-r}}$ and $ \delta _r\in {\mbox {\it Sh}(r,n-r)^{-1}}$. 

\noindent Given elements $h_{(1)}^r\in
{\mbox {\it PF}_r}$ and $h_{(2)}^r\in {\mbox {\it PF}_{n-r}}$, the element $h:=(h_{(1)}^r\times_{\mathcal P} h_{(2)}^r)\circ \delta _r\in {\mbox {\it PF}_n}$ verifies that ${\mbox {\it std}(h)}=f$.

The arguments above imply that:
$$\displaylines {
\Delta (\alpha _n(f))=\sum _{{\mbox {\it std}(h)}=f}\Delta (h)=\sum _{{\mbox {\it std}(h)}=f}(\sum_r h_{(1)}^r\otimes h_{(2)}^r)=\hfill\cr
\hfill \sum _r(\sum _{{\mbox {\it std}(h_{(i)}^r)}=f_{(i)}^r}h_{(1)}^r\otimes h_{(2)}^r)=\sum _r \alpha _r(f_{(1)}^r)\otimes \alpha _{n-r}(f_{(2)}^r),\cr }$$
which proves that $\alpha $ is a coalgebra homomorphism. \end{proof}
\medskip

Clearly, since any surjective map is a parking function, there exists the natural inclusion homomorphism $\iota : {\mbox {\bf ST}}\hookrightarrow {\mbox {\bf PQSym}^*}$, but $\iota $ is not a coalgebra homomorphism. For instance, the element $(1,1,2)$ is primitive in ${\mbox {\bf PQSym}^*}$. An element $x\in  {\mbox {\bf ST}_n}$ is such that $\iota_n (x)=\alpha_n(x)$ if, and only if, $x$ is a permutation.
\medskip

Note that $( {\mbox {\bf PQSym}^*},\times_{\mathcal P} )$ is an associative algebra, too. If we denote by ${\mbox {\it PIrr}_n}$ the subset of $\bigcup _{n\geq 1}{\mbox {\it PF}_n}$ of all parking functions $f$ such that there do not exist $f_1\in
{\mbox {\it PF}_i}$ and $f_2\in{\mbox {\it PF}_{n-i}}$ with $f=f_1\times_{\mathcal P} f_2$ and $1\leq i\leq n-1$. So, as a vector space ${\mbox {\bf PQSym}^*}$ is isomorphic to $T(\KK [\bigcup _{n\geq 1}{\mbox {\it PIrr}_n}])$, which implies that the space of primitive elements of ${\mbox {\bf PQSym}^*}$ of degree $n$ has dimension ${\mbox {\it PIrr}_n}$, for $n\geq 1$.

\subsection{Multipermutations}

In \cite{LamPyl}, T. Lam and P. Pylyavskyy  define a {\it big multi-permutation} or ${\mathcal M}$-{\it permutation} of $n$ as an ordered partition $(B_1,\dots ,B_m)$ of $n$ such that if an element $i$, $1\leq i\leq n-1$, belongs to the block $B_j$, then $i+1\notin B_j$. The set of ${\mathcal M}$-permutations of $n$ is denoted $S_n^{\mathcal M}$.
\medskip

The element $B=[(1,4,6), (2,7), (3,5)]$ is a ${\mathcal M}$-permutation of $7$, while $D=[(1,6,7), (2,3), 5, 4]$ is not.

Let $W=(W_1,\dots ,W_r)$ be an ordered partition of $n$, the  ${\mathcal M}$-standardization of $W$ is the big multi-permutation ${\mbox {\it std}_{\mathcal M}}(W)$ obtained by:
\begin{enumerate}\item delete $i+1$ if both $i$ and $i+1$ belong to the same block $W_j$,
\item if $i$ does not appear in any block obtained applying the rule above, then reduce all numbers larger than $i$ in $(1)$.
\end{enumerate}
\medskip

For example ${\mbox {\it std}_{\mathcal M}}[(1,6,7),(2,3),5,4]=[(1,5),2,4,3]$.

 Let $J$ be a subset of $[n]$ and let $B\in S_n^{\mathcal M}$, the {\it restriction} $B\vert _{J}$ of $B$ to $J$ is the intersection $B$ with $J$. If 
$B=[(1,4,6), (2,7), (3,5)]$ and $J=\{ 1,2,4,6\}$, then $B\vert _J=[(1,4,6),2]$. Let $B=[(i_1^1,\dots ,i_{r_1}^1),\dots ,(i_1^l,\dots ,i_{r_l}^l)]$ be a big multi-permutation, for any integer $k$ we denote by $B+k$ the ordered partition $[(i_1^1+k,\dots ,i_{r_1}^1+k),\dots ,(i_1^l+k,\dots ,i_{r_l}^l+k)]$.
\medskip

In \cite{LamPyl}, the authors define an algebra structure on the vector space ${\mathcal M}{\mbox {\it MR}}$ spanned by the set of all ${\mathcal M}$-permutations, as follows:
$$B\bullet D=\sum W, \ {\rm for}\ B\in S_n^{\mathcal M}\ {\rm and}\ D\in S_m^{\mathcal M},$$
where the sum is taken over :
\begin{enumerate}\item all $W\in S_{n+m}^{\mathcal M}$ such that $W\vert _{[n]}=B$ and ${\mbox {\it std}_{\mathcal M}}(W\vert _{[m]+n})=D$,
\item all $W\in S_{n+m-1}^{\mathcal M}$ such that $W\vert _{[n]}=B$ and ${\mbox {\it std}_ {\mathcal M}}(W\vert _{[m]+n-1})=D$.\end{enumerate}

For example,
$$\displaylines {
[(13),2]\bullet [2,1]=[(1,3),2,5,4]+[(1,3),(2,5),4]+[(1,3),5,2,4]+[(1,3,5), 2,4]+\hfill\cr
[5,(1,3),2,4]+[(1,3,5),(2,4)]+[5,4,(1,3),2]+[4,(1,3),2]+ [(1,3),5,4,2]+\cr
\hfill [(1,3,5),4,2]+[5,(1,3),4,2].\cr }$$
\medskip

For any ordered partition $W=(W_1,\dots ,W_l)\in S_{n+m}^{\mathcal M}$ such that $W\vert _{[n]}=B$ and ${\mbox {\it std}_ {\mathcal M}}(W\vert _{\{n+1,\dots ,n+m\}})=D$, define the integer $\cap _{B,D}^W$ as
 the number of blocks $W_j$ such that $W_j\cap [n]\neq \emptyset$ and $W_j\cap [m]+n\neq \emptyset $.
 It is immediate to check that, if $B=(B_1,\dots ,B_r)$ and $D=(D_1,\dots , D_s)$, then $\cap _{B,D}^W =r+s-l$.
\medskip

Let $B=(B_1,\dots ,B_r)$ and $D=(D_1,\dots ,D_s)$. We define binary operations $\succ _q$, $\cdot _q$ and $\prec _q$ on the space ${\mathcal M}{\mbox {\it MR}}$ as follows:\begin{enumerate}
\item $B\succ _q D:=\sum q^{\cap _{B,D}^W}W$, where the sum is taken over:\begin{itemize}
\item all $W=(W_1,\dots ,W_l)\in S_{n+m}^{\mathcal M}$ with $W\vert _{[n]}=B$ and ${\mbox {\it std}_ {\mathcal M}}(W\vert _{[m]+n})=D$, 
\item all $W\in S_{n+m-1}^{\mathcal M}$ with $W\vert _{[n]}=B$ and ${\mbox {\it std}_ {\mathcal M}}(W\vert _{[m]+n-1})=D$, \end{itemize}

\noindent such that $W_l\cap [n]=\emptyset$.
\item $B\cdot _q D:=\sum q^{\cap _{B,D}^W-1}W$, where the sum is taken over:\begin{itemize}
\item all $W=(W_1,\dots ,W_l)\in S_{n+m}^{\mathcal M}$ with $W\vert _{[n]}=B$ and ${\mbox {\it std}_ {\mathcal M}}(W\vert _{[m]+n})=D$, such that $W_l\cap [n]\neq \emptyset $ and $W_l\cap [m]+n\neq \emptyset$,
\item all $W\in S_{n+m-1}^{\mathcal M}$ with $W\vert _{[n]}=B$ and ${\mbox {\it std}_ {\mathcal M}}(W\vert _{[m]+n-1})=D$, such that $W_l\cap [n]\neq \emptyset$ and $W_l\cap 
[m]+n-1\neq \emptyset$.\end{itemize}
\item $B\prec _q D:=\sum q^{\cap _{B,D}^W}W$, where the sum is taken over:\begin{itemize}
\item all $W=(W_1,\dots ,W_l)\in S_{n+m}^{\mathcal M}$ with $W\vert _{[n]}=B$ and ${\mbox {\it std}_ {\mathcal M}}(W\vert _{[m]+n})=D$, such that $W_l\cap [m]+n=\emptyset $,
\item all $W\in S_{n+m-1}^{\mathcal M}$ with $W\vert _{[n]}=B$ and ${\mbox {\it std}_ {\mathcal M}}(W\vert _{[m]+n-1})=D$, such that $W_l\cap [m]+n-1=\emptyset$.\end{itemize}
\end{enumerate}

\noindent {\rm Let us denote by ${\mathcal M}{\mbox {\it MR}(q)}$ the space ${\mathcal M}{\mbox {\it MR}}$ equipped with the products $\succ _q$, $\cdot _q$ and $\prec _q$.}
\medskip

\begin{proposition} The data $({\mathcal M}{\mbox {\it MR}(q)},\succ _q, \cdot _q, \prec _q)$ is a $q$-tridendriform algebra.\end{proposition}

\begin{proof} Let $B=(B_1,\dots ,B_r)\in S_{n}^{\mathcal M}$, $D=(D_1,\dots ,D_s)\in S_{m}^{\mathcal M}$ and 

\noindent $E=(E_1,\dots ,E_t)\in S_{p}^{\mathcal M}$. We prove that $B\succ _q(D\succ _q E)=(B*_qD)\succ _q E$, and that $B\cdot _q(D\succ _q E)= (B\prec _q D)\cdot _q E$. The other relationships can be verified in a similar way.

We have that  $B\succ _q(D\succ _q E)=\sum _{W}\delta (W)W$, while $(B*_qD)\succ _q E=\sum _{W}\alpha(W)W$, where the both sums are taken over all the ${\mathcal M}$-permutations $W=(W_1,\dots ,W_l)$ satisfying that:
\begin{itemize}\item $W\in S_{n+m+r}^{\mathcal M}$, $W^1:=W\vert _{[n]}=B$, $W^2:=W\vert _{\{n+1,\dots ,n+m\}}=D$, $W^3:=W\vert _{\{n+m+1,\dots ,n+m+r\}}=E$ and $W_l=E_t+n+m$, 
\item $W\in S_{n+m+r-1}^{\mathcal M}$, $W^1:=W\vert _{[n]}=B$, $W^2:=W\vert _{\{n+1,\dots ,n+m\}}=D$, $W^3:=W\vert _{\{n+m,\dots ,n+m+r-1\}}=E$ and $W_l=E_t+n+m-1$,
\item $W\in S_{n+m+r-1}^{\mathcal M}$, $W^1:=W\vert _{[n]}=B$, $W^2:=W\vert _{\{n,\dots ,n+m-1\}}=D$, $W^3:=W\vert _{\{n+m,\dots ,n+m+r-1\}}=E$ and $W_l=E_t+n+m-1$,
\item $W\in S_{n+m+r-2}^{\mathcal M}$, $W^1:=W\vert _{[n]}=B$, $W^2:=W\vert _{\{n,\dots ,n+m-1\}}=D$, $W^3:=W\vert _{\{n+m-1,\dots ,n+m+r-2\}}=E$ and $W_l=E_t+n+m-2$.\end{itemize}
\medskip

We need to prove that $\alpha (W)=\delta (W)$. We give a detailed proof of it for the case $W\in S_{n+m+r}^{\mathcal M}$, the other cases are analogous.

\noindent Let $V\in S_{n+m}^{\mathcal M}$ be such that $W\vert _{[n+m]}=V$ and let $R\in S_{m+r}^{\mathcal M}$ be such that 
$W\vert _{[m+r]+n}=R+n$. We have that $\alpha (W)=\cap _{B,D}^V+\cap _{V,E}^W$, where $\cap _{B,D}^V$ is the number of blocks of $V$ which have both elements in $[n]$ and elements in $[m]+n$, while $\cap _{V,E}^W$ is the number of blocks of $W$ which have both elements in $[n+m]$ and elements in $[r]+n+m$. So, $\alpha (W)=\sum _{i=1}^l\alpha (W_i)$, where $\alpha (W_i)=$
$$\begin{cases}0,&{\rm if}\ W_i\subseteq [n]\ {\rm or}\ W_i\subseteq [m]+n\ {\rm or}\ W_i\subseteq [r]+n+m,\\
1,&{\rm if}\ W_i\ {\rm contains\ integers\ in\ exactly\ two\ sets\ of}\ [n],[m]+n\ {\rm and}\ [r]+n+m,\\
2,&{\rm if}\ W_i\ {\rm contains\ integers\ in\ all\ the\ sets}\ [n], [m]+n\ {\rm and}\ [r]+n+m.\end{cases}$$

On the other hand, $\delta (W)=\cap _{D,E}^R+\cap _{B,R}^W$, where $\cap _{D,E}^R$ is the number of blocks of $R$ which have both elements in $[m]$ and elements in $[r]+m$, while $\cap _{B,R}^W$ is the number of blocks of $W$ which have both elements in $[n]$ and elements in $[m+r]+n$, which implies that $\delta (W)=\sum _{i=1}^l\alpha (W_i)=\alpha (W)$.
\medskip

For the second equality, we have that $B\cdot _q (D\succ _qE)=\sum q^{\beta (W)}W$ and $(B\prec _q D)\cdot _q E=\sum q^{\gamma (W)}W$, where both sums are taken over all 
${\mathcal M}$-permutations $W=(W_1,\dots ,W_l)$  such that:\begin{itemize}
\item $W\in S_{n+m+r}^{\mathcal M}$, $W^1:=W\vert _{[n]}=B$, $W^2:=W\vert _{[m]+n}=D$, $W^3:=W\vert _{[r]+n+m}=E$, and $W_l=B_r\cup E_t+n+m$,
\item $W\in S_{n+m+r-1}^{\mathcal M}$, $W^1:=W\vert _{[n]}=B$, $W^2:=W\vert _{[m]+n}=D$, $W^3:=W\vert _{[r]+n+m-1}=E$, and $W_l=B_r\cup E_t+n+m-1$,
\item $W\in S_{n+m+r-1}^{\mathcal M}$, $W^1:=W\vert _{[n]}=B$, $W^2:=W\vert _{[m]+n-1}=D$, $W^3:=W\vert _{[r]+n+m-1}=E$, and $W_l=B_r\cup E_t+n+m-1$,
\item $W\in S_{n+m+r-2}^{\mathcal M}$, $W^1:=W\vert _{[n]}=B$, $W^2:=W\vert _{[m]+n-1}=D$, $W^3:=W\vert _{[r]+n+m-2}=E$, and $W_l=B_r\cup E_t+n+m-2$.\end{itemize}
To check that $\beta (W)=\gamma (W)$, for all $W$, is suffices to do a similar computation that the one in the previous case.\end{proof}
\medskip

The coproduct on the space ${\mathcal M}{\mbox MR}_+$ is defined by T. Lam and P. Pylyyavsky (see \cite{LamPyl}) as follows:
$$\Delta (B)=\sum _{[W,R]=B} {\mbox {\it std}_{\mathcal M}}(W)\otimes {\mbox {\it std}_{\mathcal M}}(R),$$
where $[W,R]$ is the union of two ordered partitions $W$ and $R$, such that $W$ is a partition of $J$ and $R$ is a partition of $K$ with $[n]=J\cup K$ and $J\cap K=\emptyset$ . 
In other words, for $B=(B_1,\dots ,B_r)$ and $0\leq j\leq r$ define $B_{\leq j}:=(B_1,\dots ,B_j)$ and $B_{>j}:=(B_{j+1},\dots ,B_r)$. The coproduct $\Delta $ on $B$ is given by:
$$\Delta (B)=\sum _{i=0}^l {\mbox {\it std}_{\mathcal M}}B_{\leq i}\otimes  {\mbox {\it std}_{\mathcal M}}B_{>i}.$$
\medskip

We have, for example, that:
$$\displaylines {
\Delta ([1])=[1]\otimes 1_{\KK}+1_{\KK}\otimes [1],\hfill\cr
\Delta ([(2)(1)])=[(2)(1)]\otimes 1_{\KK}+[1]\otimes [1]+1_{\KK}\otimes [(2)(1)],\hfill\cr
\Delta ([(13)(2)])=[(13)(2)]\otimes 1_{\KK}+[1]\otimes [1]+1_{\KK}\otimes[(13)(2)],\hfill \cr}$$
in the last example, note that $[(13)(2)]=[[(13)] ,[2]]$ and ${\mbox {\it std}_{\mathcal M}}[(13)]=[1]={\mbox {\it std}_{\mathcal M}}[2]$.

In \cite{LamPyl} the authors prove that $({\mathcal M}{\mbox MR}(1),*_1,\Delta )$ is a bialgebra. We want to show that ${\mathcal M}{\mbox {\it MR}(q)}$ equipped with the coproduct $\Delta$ is a quotient of the $q$-tridendriform bialgebra  ${\mbox {\bf ST}(q)}$.

let  $\varphi$ be the map from the set $\bigcup _{n\geq 1}{\mbox {\bf ST}_n}$ of all surjections to the set $\bigcup _{n\geq 1}S_n^{\mathcal M}$ of ${\mathcal M}$-permutations, which sends $f\in {\mbox {\bf ST}_n}$ to the element 

\noindent ${\mbox {\it std}_{\mathcal M}}[(f^{-1}(1)),\dots ,(f^{-1}(n))]$. For example, if $f=(2,3,3,6,1,5,1,2,4)$ then $$\varphi (f)={\mbox {\it std}_{\mathcal M}}[(5,7),(1,8),(2,3),9,6,4]=[(4,6),(1,7),2,8,5,3].$$ Note that $\varphi $ is surjective and does not respect the graduation.

\begin{remark}\label{standard} {\rm Let $f\in {\mbox {\bf ST}_n}$ be a surjection, and let $1\leq l\leq n$ be such that ${\mbox {\it std}_{\mathcal M}}[(f^{-1}(1)),\dots ,(f^{-1}(n))]\in S_l^{\mathcal M}$, then there exists a unique ${\overline f}\in {\mbox {\bf ST}_l}$ such that ${\mbox {\it std}_{\mathcal M}}[(f^{-1}(1)),\dots ,(f^{-1}(n))]={\mbox {\it std}_{\mathcal M}}[({\overline f}^{-1}(1)),\dots ,({\overline f}^{-1}(l))]$. Moreover, for any map ${\overline h}:\{1,\dots ,l\}\longrightarrow \{1,\dots ,r\}$ such that ${\mbox {\it std}({\overline h})}={\overline f}$, there exist a unique $h\in {\mathcal F}_n$ such that:\begin{enumerate}
\item  $({\overline h}(1),\dots ,{\overline h}(l))$ is obtained from $(h(1),\dots ,h(n))$ by eliminating all integers $h(i)$ which are equal to $h(i-1)$, for $1<i\leq n$, 
\item ${\mbox {\it std}(h)}=f$.\end{enumerate}}\end{remark}

For example, if $f=(1,2,2,3,1,4)$, then ${\overline f}=(1,2,3,1,4)$. Take ${\overline h}=(4,6,7,4,9)$, we get that $h=(4,6,6,7,4,9)$.

Applying Remark \ref{standard} we are able to prove the following result.

\begin{theorem}\label{final} For any pair of elements $f\in {\mbox {\bf ST}_n}$ and $g\in {\mbox {\bf ST}_m}$ we have that:\begin{enumerate}
\item $\varphi (f\prec_q g)=\varphi(f)\prec _q \varphi (g)$,
\item $\varphi (f\cdot _q g)=\varphi(f)\cdot _q \varphi (g)$,
\item $\varphi (f\succ _q g)=\varphi(f)\succ _q \varphi (g)$,
\item $\Delta (\varphi (f))=(\varphi \otimes \varphi )(\Delta (f))$.\end{enumerate}
\end{theorem}

\begin{proof} If $h$ and $k$ are two maps such that $hk\in {\mbox {\bf ST}_{n+m}}$, ${\mbox {\it std}(h)}=f$ and ${\mbox {\it std}(k)}=g$, then
$$\displaylines {
{\mbox {\it std}_{\mathcal M}}[(h^{-1}(1)),\dots ,(h^{-1}(r))]={\mbox {\it std}_{\mathcal M}}[(f^{-1}(1)),\dots ,(f^{-1}(n))],\hfill\cr
{\mbox {\it std}_{\mathcal M}}[(k^{-1}(1)),\dots ,(k^{-1}(r)]={\mbox {\it std}_{\mathcal M}}[(g^{-1}(1)),\dots ,(g^{-1}(m)],\hfill\cr }$$
where ${\mbox {\it max}(hk)}=r\leq n+m$.

Suppose that  ${\mbox {\it max}(h)}>{\mbox {\it max}(k)}$, we have that:
$$[((hk)^{-1}(1)),\dots ,((hk)^{-1}(r))]=[(h^{-1}(1)\cup (k^{-1}(1)+n)),\dots ,(h^{-1}(r-1)\cup (k^{-1}(r-1)+n)),(h^{-1}(r))],$$
where $(h^{-1}(i)\cup (k^{-1}(i)+n))$ denotes the disjoint union of the sets $h^{-1}(i)$ and $k^{-1}(i)+n$, for $1\leq i\leq r-1$.
The standardization ${\mbox {\it std}_{\mathcal M}}[((hk)^{-1}(1)),\dots ,((hk)^{-1}(r))]$ is a ${\mathcal M}$-permutation $W=(W_1,\dots ,W_r)$ satisfying that:\begin{enumerate}
\item if $h(n)\neq k(1)$, then $W\vert _{[n]}={\mbox {\it std}_{\mathcal M}}[(f^{-1}(1)),\dots ,(f^{-1}(n))]$, $W\vert _{[m]+n}={\mbox {\it std}_{\mathcal M}}[(g^{-1}(1)+n),\dots ,(g^{-1}(m)+n)]$ and $W_r\cap ([m]+n)=\emptyset $,
\item if $h(n)= k(1)$, then $W\vert _{[n]}={\mbox {\it std}_{\mathcal M}}[(f^{-1}(1)),\dots ,(f^{-1}(n))]$, $W\vert _{[m]+n-1}={\mbox {\it std}_{\mathcal M}}[(g^{-1}(1)+n-1),\dots ,(g^{-1}(m)+n-1)]$ and $W_r\cap ([m]+n-1)=\emptyset $.\end{enumerate}

Conversely, let $W=(W_1,\dots ,W_r)$ be a ${\mathcal M}$-permutation such that $W\vert _{[n]}=\varphi (f)={\mbox {\it std}_{\mathcal M}}[(f^{-1}(1)),\dots ,(f^{-1}(n))]$ and $W_r\subseteq [n]$, we have that \begin{enumerate}
\item if $W\vert _{[m]+n}={\mbox {\it std}_{\mathcal M}}[(g^{-1}(1)+n),\dots ,(g^{-1}(m)+n)]$, then there exist maps ${\overline h}$ and ${\overline k}$ defined as follows:\begin{enumerate}
\item ${\overline h}(i)$ is the unique integer such that $i\in W_{{\overline h}(i)}$.
\item ${\overline k}(j)$ is the unique integer such that $j+n\in W_{{\overline k}(j)}$.\end{enumerate}
By Remark \ref{standard}, there exist unique elements $h\in {\mathcal F}_n$ and $k\in {\mathcal F}_m$ such that ${\mbox {\it std}(h)}=f$, ${\mbox {\it std}(k)}=g$ and ${\mbox {\it std}_{\mathcal M}}[((hk)^{-1}(1)),\dots ,((hk)^{-1}(r))]=W$. \item if $W\vert _{[m]+n-1}={\mbox {\it std}_{\mathcal M}}[(g^{-1}(1)+n-1),\dots ,(g^{-1}(m)+n-1)]$, then the maps ${\overline h}$ and ${\overline k}$ are defined as follows:\begin{enumerate}
\item ${\overline h}(i)$ is the unique integer such that $i\in W_{{\overline h}(i)}$.
\item ${\overline k}(j)$ is the unique integer such that $j+n-1\in W_{{\overline k}(j)}$.\end{enumerate}
Again, there exist unique elements $h\in {\mathcal F}_n$ and $k\in {\mathcal F}_m$ such that ${\mbox {\it std}(h)}=f$, ${\mbox {\it std}(k)}=g$ and ${\mbox {\it std}_{\mathcal M}}[((hk)^{-1}(1)),\dots ,((hk)^{-1}(r))]=W$. 
\end{enumerate}
Moreover, since $W_r\cap ([m]+n)=\emptyset $, we get that ${\mbox {\it max}(k)}<{\mbox {\it max}(h)}=r$ in both cases. 

We get then that $\varphi (f\prec_q g)=\varphi(f)\prec _q \varphi (g)$, the proofs of the second and third statements follow from similar arguments.
\medskip

To end the proof of the theorem we need to show that $\varphi $ is a coalgebra homomorphism. For $f\in {\mbox {\bf ST}_n}$, let ${\overline f}\in {\mbox {\bf ST}_l}$ be the unique surjection such that 
${\mbox {\it std}_{\mathcal M}}[(f^{-1}(1)),\dots ,(f^{-1}(r))]=[({\overline f}^{-1}(1)),\dots , ({\overline f}^{-1}(l))]$. It is easy to see that there exists a bijection between the set of elements $(r,f_{(1)}^k,f_{(2)}^k)$, with $0\leq k\leq n$, 
$f_{(1)}^k\in {\mbox {\bf ST}_k}$ and $f_{(2)}^k\in {\mbox {\bf ST}_{n-k}}$, such that $$f=(f_{(1)}^k\times f_{(2)}^k)\circ \delta _k,\ {\rm for\ some}\ \delta _r\in {\mbox {\it Sh}(k,n-k)},$$
and the set of elements $(j,{\overline f}_{(1)}^j,{\overline f}_{(2)}^j)$, with $0\leq j\leq l$, 
${\overline f}_{(1)}^j\in {\mbox {\bf ST}_j}$ and ${\overline f}_{(2)}^j\in {\mbox {\bf ST}_{l-j}}$, such that $${\overline f}=({\overline f}_{(1)}^j\times {\overline f}_{(2)}^j)\circ \tau _j,\ {\rm for\ some}\ \tau _j\in {\mbox {\it Sh}(j,l-j)}.$$

So, it suffices to verify that $\Delta (\varphi (f))=(\varphi \otimes \varphi )(\Delta (f))$ for $f$ satisfying that ${\mbox {\it std}_{\mathcal M}}[(f^{-1}(1)),\dots ,(f^{-1}(n))]=[(f^{-1}(1)),\dots , (f^{-1}(r))]$, that is when $f(i)\neq f(i+1)$ for $1\leq i\leq n-1$. If for some $0\leq k\leq n$ there exist $f_{(1)}^k\in {\mbox {\bf ST}_k}$, $f_{(2)}^k\in {\mbox {\bf ST}_{n-k}}$ and $\delta _k\in {\mbox {\it Sh}(k,n-k)}$ such that $f=(f_{(1)}^k\times f_{(2)}^k)\circ \delta _k$ and 
$W:=[((f_{(2)}^k)^{-1}(1)+s),\dots ,((f_{(2)}^k)^{-1}(n-k)+s)]$.

Conversely, let $[(f^{-1}(1)),\dots , (f^{-1}(r))]=[R,W]$, with $R=[R_1,\dots ,R_s]$ and $W=[W_1,\dots ,W_{r-s}]$, and suppose that $R$ is a partition of $\{i_1<\dots <i_k\}$ and $W$ is a partition of $\{j_1<\dots <j_{n-k}\}$. Define $f_{(1)}^k$ and $f_{(2)}^k$ as follows:\begin{enumerate}
\item $f_{(1)}^k(l)$ is the unique integer $1\leq f_{(1)}^k(l)\leq s$ such that $i_l\in R_{f_{(1)}^k(l)}$, for $1\leq l\leq k$,
\item $f_{(2)}^k(l)$ is the unique integer $1\leq f_{(2)}^k(l)\leq r-s$ such that $j_l\in R_{f_{(2)}^k(l)}$, for $1\leq l\leq n-k$.\end{enumerate}
It is clear that there exists a shuffle $\delta _k\in {\mbox {\it Sh}(k,n-k)}$ such that $f= (f_{(1)}^k\times f_{(2)}^k)\circ \delta _k$.
We get then that 
$$\displaylines {
(\varphi \otimes \varphi ) (\Delta (f))=\sum_k \varphi (f_{(1)}^k)\otimes \varphi (f_{(2)}^k)=\hfill\cr
\sum _k {\mbox {\it std}_{\mathcal M}}[((f_{(1)}^k)^{-1}(1)),\dots ,((f_{(1)}^k)^{-1}(k))]\otimes {\mbox {\it std}_{\mathcal M}}[((f_{(2)}^k)^{-1}(1)),\dots ,((f_{(2)}^k)^{-1}(n-k))]=\cr
\hfill \sum _{\varphi (f)=[R,W]}{\mbox {\it std}_{\mathcal M}}(R)\otimes {\mbox {\it std}_{\mathcal M}}(W)=\Delta (\varphi (f)).\cr }$$\end{proof}

As a consequence of Theorem \ref{final}, we can assert that ${\mathcal M}{\mbox {\it MR}(q)}$ is a $q$-tridendriform bialgebra.

\begin{corollary} The $q$ tridendriform algebra ${\mathcal M}{\mbox {\it MR}(q)}$ equipped with the coproduct $\Delta $ is a $q$-tridendriform bialgebra which is a quotient of ${\mbox {\bf ST}(q)}$.\end{corollary}

\end{document}